\documentclass[12pt]{amsart}

\usepackage{amssymb,bm, a4, mathrsfs}

\usepackage{tikz}

\usetikzlibrary{arrows}
\usepackage{rotating}

\newtheorem{Thm}{Theorem}[section]

\newtheorem{lem}[Thm]{Lemma}
\newtheorem{cor}[Thm]{Corollary}

\newtheorem{Cor}[Thm]{Corollary}
\newtheorem{Lem}[Thm]{Lemma}
\newtheorem{Prop}[Thm]{Proposition}

\newtheorem{them}[Thm]{Theorem}

\newtheorem{Result}[Thm]{Result}
\theoremstyle{definition}		
\newtheorem{Def}[Thm]{Definition}

\newtheorem{Ex}[Thm]{Example}
\theoremstyle{definition}

\newtheorem{Conj}[Thm]{Conjecture}

\newtheorem{qn}[Thm]{Question}

\usepackage{verbatim}
\usepackage{pgf}
\usepackage{pifont}
\usepackage{bbding}

\usepackage{color}

\usepackage[latin1]{inputenc}

\usepackage{CJK}
\usepackage{indentfirst}

\theoremstyle{definition}

\numberwithin{Thm}{section}
\numberwithin{equation}{section}

\newcommand{\F}{\mbox{{$\mathscr{F}$}}}
\newcommand{\FF}{\mbox{{$\mathcal{F}$}}}

\newcommand{\N}{\mbox{{$\mathbb{N}$}}}

\begin{document}
\title{Coherency for monoids and purity for their acts}


\author{Yang Dandan}

\address{School of Mathematics and Statistics, Xidian University, Xi'an 710071, P.R. China}

\email{ddyang@xidian.edu.cn}

\author{Victoria Gould}

\address{Department of Mathematics, University of York, Heslington, York YO10 5DD, UK}

\email{victoria.gould@york.ac.uk}


\thanks{This work was supported by Grant No. EP/V002953/1 of the Engineering and Physical Sciences Research Council,  Grant No. 12171380 of the National Natural Science Foundation of China, by Grant No. QTZX2182 of the Fundamental Research Funds for the Central Universities and by Grant No. 2023-JC-JQ-04 of the Natural Science Basic Research Program of Shaanxi. }
\maketitle

\begin{abstract}
This article examines the three-way relationship between right coherency of a monoid $S$, solutions of equations over $S$-acts, and injectivity properties of $S$-acts.
A monoid $S$  is {\em right coherent} if every finitely generated subact of every finitely presented (right) $S$-act itself has a finite presentation.  {\em Purity properties} of
an $S$-act $A$ may either be expressed in terms of solutions in $A$  of certain consistent sets of equations over $A$, or in terms of  injectivity properties.
For example, an  $S$-act $A$ is {\it absolutely pure (almost pure)} if every finite consistent set of equations over $A$ (in one variable) has a solution in $A$. Equivalently,  $A$ is {\it absolutely pure (almost pure)}
if it is injective with respect to   inclusions of  finitely generated subacts into finitely presented (monogenic  finitely presented) $S$-acts.

Our first main result shows that for a right  coherent monoid  $S$
the classes of  almost pure and absolutely pure $S$-acts coincide.
Our second main result is that a monoid $S$ is right coherent if and only if the classes of mfp-pure and
absolutely pure $S$-acts coincide:  an $S$-act is
 {\em mfp-pure}
if it is injective with respect to  inclusions of   finitely presented subacts into  monogenic finitely presented $S$-acts. We give  specific  examples of  monoids $S$ that are not right coherent
yet are such that the classes of
almost pure and absolutely pure $S$-acts coincide.  Finally we give a condition on a monoid $S$ for all almost pure $S$-acts to be absolutely pure in terms of finitely presented $S$-acts, their finitely generated subacts, and certain canonical extensions.
\end{abstract}

\section{Introduction and preliminaries}\label{sec:intro}

 This article is a contribution to the study of coherency for monoids. Specifically, it concerns  the relationship between coherency of a monoid and purity properties of its acts.
Let $S$ be a monoid  with identity 1.
Coherency of $S$ may be defined in terms of its $S$-acts.
 A {\em right $S$-act}
is a set $A$ together with a map $A\times S\rightarrow A$, where $ (a,s)\mapsto as$, such that for all $a\in A$ and $s,t\in S$ we have $a1=a$ and $a(st)=(as)t$.
{\em Left $S$-acts} are defined dually; by `$S$-act' we will  mean by default `right $S$-act', with the corresponding convention for $R$-modules over a ring $R$.
 An $S$-act
is a representation of  $S$ by mappings  of a set, analogously to the way in which an  $R$-module is a representation of a ring $R$ by homomorphisms of an abelian group. The theory of $S$-acts  both intertwines with that of $R$-modules, and pulls apart from it,  a phenomenon emphasised by this article.

 A monoid $S$  is {\em right coherent} if  every finitely generated subact  of every finitely presented  $S$-act is finitely presented. This definition is analogous to that for a ring $R$, where the notion of $S$-act is replaced by that of  $R$-module. For both monoids and rings, right coherency is an important finitary  condition, that is, one certainly satisfied by all finite monoids or rings, and is strictly weaker than that of being right noetherian
\cite{normak:1977,rotman:1979}. In fact, a ring $R$ is right coherent if and only if every finitely generated right ideal of $R$ has a finite presentation \cite{chase:1960}. The  corresponding  statement is not true for $S$-acts, the free inverse monoid providing a counter-example \cite{gould:2017a}.  Essentially this split in the theories is due to the fact that for $S$-acts, congruences are not determined by subacts.  Moreover,  right coherency of $R$ is equivalent to the property that products of flat {\em left} $R$-modules are flat \cite{chase:1960}. Again, we do not have that tool to use for $S$-acts, although some partial results are known \cite{gould:1992}. Here \cite{bulmanfleming:1990, khosravi:2009} are also relevant, since they consider closure properties of the classes of flat  {\em left} $S$-acts, and use this to define a related notion of  coherency.

 Although a very natural  property,  it transpires that right coherency for monoids is difficult to pin down. Even with the  aid of a Chase-type condition as in Theorem~\ref{thm:coherency},   it can be hard to ascertain  whether or not a given monoid is right coherent. Nevertheless, right coherency (or not) of monoids in a number of important classes has been determined \cite{gould:1992,gould:2017,gould:2017a}. The interaction between coherency and standard algebraic properties is subtle \cite{dandan:2020}.

Coherency for both monoids and rings is related to the model theory of their acts and modules.  In 1976 Wheeler \cite{wheeler:1976} defined a coherent theory for a first order language. A theory of  $S$-acts  or $R$-modules  is coherent in Wheeler's sense if and only if $S$ or $R$ is right coherent in our sense, and this is equivalent to their classes of existentially closed $S$-acts or $R$-modules being first order axiomatisable \cite{gould:1986,wheeler:1976}. Existential closure refers to the existence of solutions of finite consistent sets of equations and inequations. In this article we will be examining the relationship between right coherency and {\em equations},  the latter providing one approach to the properties we refer to as purity properties.

Given an $S$-act $A$ an {equation} over $A$ has one of the following three forms:
 $xs=xt,\, xs=yt$ or $ xs=a$ where $x, y$ are {\it variables}, $s, t\in S$ and $a\in A$ is a {\it constant}.
We will set up our notation for equations over $A$ more formally in  Section~\ref{sec:equations}. A set $\Sigma$ of equations over $A$ is {\it consistent} if $\Sigma$ has a solution in some $S$-act $B$ containing $A$. We are concerned with the question of when a  consistent set $\Sigma$ of equations over $A$, of a particular form, has a  solution in $A$. This leads us to so-called purity notions for an $S$-act. We now outline the main ones of our concern.

 An $S$-act $A$ is {\em absolutely pure} if every finite consistent set of equations with constants from $A$ has a solution in $A$. An $S$-act $A$ is {\it almost pure} if every finite consistent set of equations in one variable with constants from $A$ has a solution in $A$.
  These and other  notions of purity may equivalently be phrased in terms of completion of diagrams,   as weak versions of injectivity, whence the terminology arises.

We recall that an $S$-act  $A$ is {\em injective} if  any diagram of $S$-acts  and $S$-morphisms of the form on the left
\begin{center}

\begin{tikzpicture}

\node at (0,0) {$C$};
\node at (3,0) {$B$};
\node at (3,2) {$A$};

\node at (3.2,1) {$\theta$};


\node at (0,-.5) {\phantom{f.p.}};

\node at (3,-.5) {\phantom{f.g.}};

\node at (0,-1) {\phantom{cyclic}};

\draw [left hook-> , thick] (2.8,0) -- (.3,0);

\draw [->, thick] (3,0.2) -- (3,1.7);


\node at (6,0) {$C$};
\node at (9,0) {$B$};
\node at (9,2) {$A$};

\node at (9.2,1) {$\theta$};

\node at (7.5,1.3) {$\overline{\theta}$};

\node at (6,-.5) {\phantom{f.p.}};

\node at (9,-.5) {\phantom{f.g.}};

\node at (6,-1) {\phantom{cyclic}};

\draw [left hook-> , thick] (8.8,0) -- (6.3,0);

\draw [->, thick] (9,0.2) -- (9,1.7);

\draw [->, thick] (6.2,0.2) -- (8.8,1.7);

\end{tikzpicture}\end{center}
 \noindent may be completed via an $S$-morphism $\overline{\theta}$ as on the right. It is known that an  $S$-act $A$ is absolutely pure (almost pure) if and only if any diagram on the left, where $C$ is finitely presented (and monogenic) and $B$ is finitely generated,  can be completed as on the right (see \cite[Proposition 3.8]{gould:1985} and \cite[Proposition 3.2]{gould:1987}\footnote{In the latter, empty acts were not allowed, hence the slightly different wording.}).   By imposing the condition that $B$ {\em and} $C$ are finitely presented and $C$ is monogenic we obtain the notion we call mfp-purity.  We explain in Section~\ref{sec:equations} how  mfp-purity may be  correspondingly phrased in terms of equations. Analogous notions and similar observations are true for $R$-modules (see, for example,  \cite{stenstrom:1970, megibben:1970}, and also \cite{lu:2012}).

We denote by $\mathcal{A}^{fp}_S(1),\mathcal{A}_S(1)$ and by $\mathcal{A}_S(\aleph_0)$ the classes of mfp-pure, almost pure and absolutely pure $S$-acts, respectively.  Clearly, any  absolutely pure $S$-act is almost pure and any almost pure $S$-act is mfp-pure,   that is,\[ \mathcal{A}_S(\aleph_0)\subseteq \mathcal{A}_S(1)\subseteq \mathcal{A}^{fp}_S(1).\]
 The  question of the converse inclusions motivates much of this paper; we demonstrate that the answers are intimately related to the notion of right coherency.

\begin{qn} \label{qn:ap} For which monoids $S$ is:
\begin{enumerate}
\item $\mathcal{A}_S(\aleph_0)=\mathcal{A}_S(1)$?
\item  $\mathcal{A}_S(1)=\mathcal{A}^{fp}_S(1)$?
\item  $\mathcal{A}_S(\aleph_0)=\mathcal{A}^{fp}_S(1)$?
\end{enumerate} \end{qn}

 It is pertinent to pose Question~\ref{qn:ap}, for the following reasons. Concerning (1), we know that if  {\em all} $S$-acts are almost pure, then {\em all} $S$-acts are absolutely pure \cite{gould:1987}. Second,
an $S$-act $A$ is injective if and only if all consistent sets of equations over $A$ have a solution in $A$
\cite[Proposition 3.10]{gould:1985} and by the Skornjakov-Baer
Criterion \cite{kkm:2000}, this is equivalent to all consistent sets of equations in one variable over $A$ having a solution in $A$.  However, the proof of the Skornjakov-Baer
Criterion  uses arguments that do not work in our  case
of finite sets of equations. From the proof of  \cite[Theorem 4]{megibben:1970},  for a right coherent ring any almost pure module is absolutely pure.  However,  the full solution to the corresponding question to (1) is  still   open for $R$-modules, as well as for $S$-acts. It is worth noting that for some other classes of algebras, with very different signatures, (1) has a positive answer. In particular, if a  group
$G$ has the property that any finite consistent set of equations in one variable with constants from $G$ has a solution in $G$, then it has the property that any finite consistent set of equations in any (finite) number of  variables with constants from $G$ has a solution in $G$;  the same is true for semigroups \cite{levin:64,levin:70}.  These results for semigroups and groups use a property of extensions that does not hold for $S$-acts in general. 

Concerning (2) and (3), by very definition,  a right coherent monoid is such that $\mathcal{A}_S(1)=\mathcal{A}^{fp}_S(1)$. The situation for rings gives us some  pointers to the conjecture that only right coherent monoids will give this equality. The article \cite{lu:2012} demonstrates that all IFP-injective $R$-modules are  absolutely pure  if and only if  $R$ is right coherent. Here the property of being IFP-injective is closely analogous to mfp-purity.   We note that the classical work for rings, as may be found in \cite{stenstrom:1970, megibben:1970, lu:2012}, and other articles, use ring theoretic techniques and  results, including the correspondence with flatness properties, that are not valid for monoids.

We do not fully answer Question~\ref{qn:ap}(1) but we are able to show the class of monoids $S$ such that   $\mathcal{A}_S(\aleph_0)=\mathcal{A}_S(1)$ properly contains the class of right coherent monoids.  It follows that the property of a monoid that $\mathcal{A}_S(\aleph_0)=\mathcal{A}_S(1)$ is a {\em finitary property}, that is, one satisfied by all finite monoids.  On the other hand we fully answer Question~\ref{qn:ap}(2) and (3), with the classes in question being precisely that of  right coherent monoids.
To prove our results, we establish and utilise two pieces  of machinery.  One enables us to pass smoothly between the equational approach to purity and weak injectivity properties.  The other involves constructing, for any $S$-act $A$ and a given purity property, a canonical extension of
$A$ having that property.

We proceed as  follows.  In Section~\ref{sec:pre} we set up our notation and give preliminary results that will be used throughout. In Section~\ref{sec:equations} we introduce the notion of a {\em frame}  $\mathcal{F}$ of a set of equations,  of a {\em frame set} $\F$, and of $\F$-{\em purity}. This allows us to  build the aforementioned machinery to  fully delineate the passage between  purity properties of an $S$-act, and weak injectivity. The results above for almost and absolutely pure $S$-acts are special cases.

 Section~\ref{sec:coherent} contains our first main result, motivated by  Question~\ref{qn:ap}(1).

 \begin{Result}\label{res:coherent} (cf. Theorem~\ref{coherent monoids}). Let $S$ be a right coherent monoid. Then $\mathcal{A}_S(\aleph_0)=\mathcal{A}_S(1)$.
 \end{Result}

To answer Question~\ref{qn:ap}(2) in Section~\ref{sec:const}  we build, for a frame set $\F$ and an $S$-act $A$, an   $\F$-pure extension $A(\F)$ of $A$ that is canonical in the sense $A$ is $\F$-pure if and only if $A$ is a retract of $A(\F)$.
This is our second promised piece of machinery. In Section~\ref{sec:mfp} it is utilised to prove our second main result, which completely answers  Questions~\ref{qn:ap}(2) and (3).

\begin{Result}\label{res:mfp:main} (cf. Theorem~\ref{mfp:main}). A monoid $S$ is right coherent if and only if $\mathcal{A}_S(1)=\mathcal{A}^{fp}_S(1)$ if and only if  $\mathcal{A}_S(\aleph_0)=\mathcal{A}^{fp}_S(1)$.
\end{Result}

An immediate question is whether or not  right coherency is a necessary condition for  $\mathcal{A}_S(\aleph_0)=\mathcal{A}_S(1)$?
 The answer is no. It is easy to see that if $S$ has the property that every finitely generated $S$-act embeds into a monogenic act, then again  $\mathcal{A}_S(\aleph_0)=\mathcal{A}_S(1)$. Such monoids are somewhat special; in particular, they cannot have zeros. Our next  result, in Section~\ref{sec:ex}, hangs on delicate analysis of  a particular monoid, named the Fountain monoid.

  \begin{Result}\label{res:fountain}(cf. Theorem~\ref{fountain monoid}) There exists a monoid $S$ that is not right coherent, is such that not every finitely generated $S$-act embeds into a monogenic act, but $\mathcal{A}_S(\aleph_0)=\mathcal{A}_S(1)$.\end{Result}

   We believe our example is one of a broader class, and we pose the corresponding problem at the end of Section~\ref{sec:ex}.

Finally, in Section~\ref{sec:ap=ap}, we use the machinery developed in Section~\ref{sec:const}  to  give a condition on $S$  for
  $\mathcal{A}_S(\aleph_0)=\mathcal{A}_S(1)$ in terms of finitely presented $S$-acts, their finitely generated $S$-subacts, and their canonical extensions.
The question of whether or not $\mathcal{A}_S(\aleph_0)=\mathcal{A}_S(1)$ for all monoids $S$ is still open, as it is for $R$-modules, although we conjecture the answer will be negative.

We attempt to keep this paper as self-contained as possible. For further details we refer the reader to \cite{howie:1995} for background in semigroup theory, and to \cite{kkm:2000} for  information on monoid acts.

\section{Preliminaries}\label{sec:pre}

The aim of this section is to set up notation and then proceed to preliminary results,  which  will be used throughout the article.

\subsection{The category of $S$-acts}\label{sub:sacts} Let $S$ be a monoid  with identity 1. We recall that a {\em right $S$-act} is a set $A$ together with a map $$A\times S\rightarrow A, (a,s)\mapsto as$$ such that
for all $a\in A$ and $s,t\in S$ we have $a1=a$ and $a(st)=(as)t$. Naturally, we may define left $S$-acts in a dual manner, but in this article all $S$-acts will be right $S$-acts, and for convenience we will refer to them simply as {\em $S$-acts}.  Note that we allow $A=\emptyset$. If $A$ is an $S$-act, then there is a monoid morphism from $S$ to the full transformation monoid $\mathcal{T}_A$ on $A$, taking $s$ to $\rho_s$, where $a\rho_s=as$. Conversely, any morphism $\varphi$ from $S$ to the full transformation monoid $\mathcal{T}_B$ on a set $B$ makes $B$ into an $S$-act
by setting $bs:= b(s\varphi)$.   The study of $S$-acts is, therefore, that of representations of the monoid $S$ by mappings of sets. Not surprisingly, in view of the natural way in which they arise, $S$-acts come under a plethora of names ($S$-sets, $S$-polygons, $S$-systems, to name a few). We note that any  unary algebra may be regarded as   an act, for example, over the free monogenic monoid.

For any monoid $S$ the class of all $S$-acts forms a variety of universal algebras, where the basic operations are the unary operations $\{ \rho_s:s\in S\}$. We refer to an algebra morphism in this variety as an {\em $S$-morphism}. It follows that a function  $\phi:A\rightarrow B$, where $A$ and $B$ are $S$-acts,  is an $S$-morphism if $(as)\phi=(a\phi)s$  for all $a\in A, s\in S$. In the standard way we have a category, the objects of which are $S$-acts and the morphisms of which are $S$-morphisms. A subset $B$ of an $S$-act $A$ is a {\em subact} if $bs\in B$ for all $b\in B, s\in S$.  If $B$ is a subact of $A$ then a {\em retract} from $A$ to $B$ is  an
$S$-morphism $\varphi:A\rightarrow B$ such that $\varphi|_{B}$ is the identity map $1_B$ of $B$. The set of subacts of $A$ is well behaved  in the sense it is closed under unions and intersections. In fact, a disjoint union of any $S$-acts is again an $S$-act in an obvious way.  Any right  ideal of $S$ is a right  $S$-act so  $S$ itself is a right  $S$-act.  That $S$ is the free monogenic (i.e. single generated, or cyclic) $S$-act follows from the below.

An $S$-act $F$ is {\em free on a set $X$} if there is a map $\iota:X\rightarrow F$ such that for any $S$-act $A$ and map $f:X\rightarrow A$ there is a unique $S$-morphism $\varphi:F\rightarrow A$ such that $\iota\varphi=f$. Since $S$-acts form a variety  the free $S$-act  on $X$ exists. It has a transparent structure, which we now describe.
 Put
 \[F_S(X)=X\times S :=\bigcup_{x\in X}xs\]
 where we make the (convenient)  identifications $(x,s):=xs $
 and $(x,1):= x$.
 Define an action of $S$ on $F_S(X)$ by $(xs)t=x(st)$.  Then it is easily seen that $F_S(X)$ is the free $S$-act on $X$ where $x\iota= x$.
 Note that for any $s,t\in S$ and $x, y\in X$, we have that $xs=yt$ if and only if $x=y$ and $s=t$.

Morphic images of $S$-acts are obtained by factoring out by the appropriate notion of congruence. Let $A$ be an $S$-act. A {\em congruence} $\rho$ on $A$ is an equivalence relation such that for any $a,b\in A$ with $a\,\rho\, b$ and any $s\in S$ we have $as\,\rho\, bs$.  We refer to a congruence on $S$ regarded as an $S$-act as a {\em right congruence} on $S$. Denoting the equivalence class of $a\in A$ by $[a]$ we have
\[A/\rho=\{ [a]:a\in A\}\]
 is an $S$-act under the action $[a]s=[as]$. It is called   the {\em quotient} of $A$ by $\rho$. The map  $\nu:A\rightarrow A/\rho$ is then the  natural $S$-morphism with kernel $\rho$.
For $H \subseteq A\times A$  the {\em congruence generated by $H$}, denoted by $\langle H\rangle$,  is the least congruence on $A$ containing $H$. Without further remark we assume that $H$ is always symmetric. An explicit formula for  $\langle H\rangle$   is obtained as follows.

\begin{lem}\label{lem:gen} \cite{howie:1995} Let $A$ be an $S$-act and let  $H\subseteq A\times A$. Then for any $a,b\in A$ we have $a\,\langle H\rangle\, b$ if and only if $a=b$ or there exists a  sequence
\[a=c_1t_1, d_1t_1=c_2t_2, \cdots, d_nt_n=b\] where $t_i\in S$ and $(c_i, d_i)\in H$ for all $1\leq i\leq n$.
\end{lem}
A sequence as above will be referred to as an {\em $H$-sequence of length $n$}.
 We interpret $a=b$ as belonging to an $H$-sequence of length $0$.

The next definitions are merely the translations of general algebraic notions to our context.

\begin{Def}\label{def:fgfp} An $S$-act $A$ is {\em finitely generated} if $A$ is isomorphic to $F_S(X)/\rho$ for some finite set $X$ and  congruence $\rho$.
\end{Def}

It is clear that a non-empty act $A$ is finitely generated if and only if   for some $n\in \mathbb{N}$ and $a_i\in A$,
$1\leq i\leq n$, we have
$A=a_1S\cup\cdots \cup a_nS$.
 Similarly, $A$ is monogenic if and only if
$A=aS$ for some $a\in S$.

\begin{Def}\label{def:fgfp} An $S$-act $A$ is {\em finitely presented} if $A$ is isomorphic to $F_S(X)/\rho$ for some finite set $X$ and finitely generated congruence $\rho$ on $F_S(X)$.
\end{Def}

We remark that being finitely presented is not dependent on the chosen set of generators.

\subsection{Right coherency}

The notion of coherency is a central one to this article. We recall from Section~\ref{sec:intro}:

\begin{Def}\label{def:coherency} A monoid $S$ is {\em right coherent} if every finitely generated subact of any finitely presented $S$-act is itself finitely presented.
\end{Def}

To test whether a specific monoid is right coherent we usually make use of the following, which is reminiscent of the result of Chase for rings \cite{chase:1960}.

\begin{them}\label{thm:coherency}\cite{gould:1992}
The following are equivalent for a monoid $S$:

(i) $S$ is right coherent;

(ii) any finitely generated subact of $S/\rho$, where $\rho$ is a finitely generated right congruence on $S$, is  finitely presented;

(iii) for any finitely generated right congruence $\rho$ on $S$ and any $s, t\in S$:

\indent\indent (1) the subact $(s\rho)S\cap (t\rho)S$ of the right $S$-act $S/\rho$ is finitely generated;

\indent\indent (2) the annihilator $$\bold{r}(s\rho)=\{(u, v)\in S\times S: su~\rho~sv\}$$ \indent\indent\indent is a finitely generated right congruence on $S$;

(iv) for any finite set $X$ and finitely generated right congruence $\rho$ on $F_S(X)$ and any $a, b\in F_S(X)$:

\indent\indent (1) the subact $(a\rho)S\cap (b\rho)S$ of  $F_S(X)/\rho$ is finitely generated;

\indent\indent (2) the annihilator $$\bold{r}(a\rho)=\{(u, v)\in S\times S: au~\rho~av\}$$ \indent\indent\indent  is a finitely generated right congruence on $S$.
\end{them}

 It is known that groups, monoid semilattices (regarded as commutative monoids of idempotents), Clifford monoids (monoid semilattices of groups), free commutative and free monoids are all (right) coherent \cite{gould:1992,gould:2017}. Regular monoids for which every right ideal is finitely generated are right coherent \cite{gould:2017}, where a monoid $S$  is {\em regular} if for all $a\in S$ there exists $x\in S$ such that $a=axa$. A monoid is {\em inverse} if it is regular and its idempotents commute.  Groups, semilattices, and Clifford monoids are all inverse, but not all inverse monoids are right coherent; for example, the free inverse monoid on a set with more than one generator is not right coherent \cite{gould:2017a}.

\section{Equations over $S$-acts}\label{sec:equations}
 As  promised in Section ~\ref{sec:intro}, we now formally set up our notation for   equations.
 We then build machinery that will allow us to pass between solutions of consistent sets of equations  and weak injectivity properties of an act. In order that our techniques have the widest application, we take care over the exact forms of equations, introducing the notions of equation form, frame, and frame set.

In what follows  $X$ is a non-empty set, but we do not always mention $X$ explicitly. The reason is that elements of $X$ will ultimately correspond to variables, the exact labelling of which is usually unimportant.

\begin{Def}\label{def:forms} An {\em equation form (with variables from $X$)}  is an element $f=f_S(X)$  of
\[\big(F_S(X)\times F_S(X)\big)\cup F_S(X).\]
If $f\in F_S(X)\times F_S(X)$ then we say {\em $f$ has type $2$}; if $f\in F_S(X)$ then we say {\em $f$ has type $1$}.
\end{Def}

\begin{Def} \label{def:eqn} Let $A$ be an $S$-act and let $f=f_S(X)$ be an equation form.  An
{\em equation over $A$ with equation form $f$ (and variables from $X$)} is an expression
\[\begin{array}{rcll}
xs&=&yt& \mbox{ if }f\mbox{ is }(xs,yt)\\
xs&=&a&\mbox{ where }a\in A \mbox{ if }f\mbox{ is }xs.\end{array}\]
\end{Def}

 Notice that an equation form of type 2 corresponds to a single equation, whereas a form of type 1 corresponds to different equations, which depend on a choice of an $S$-act $A$ and $a\in A$. It is also worth emphasising that equations over $A$  essentially come in  three types:
\[ xs=yt,\, xs=xt \mbox{ or }xs=a\]
where $x\neq y\in X, s\in S$ and $ a\in A$. In expressions of this kind the roles of $x,y,s,t,a$ etc. will be implicit.  Note that at one and the same time we may regard $x\in X$ as an element of the free $S$-act $F_S(X)$ and as a variable to be substituted by an element of an $S$-act.

 If $\Sigma=\Sigma(X)$ is  a set of equations over an $S$-act $A$ then we  do not insist that every element of $X$ appears in at least one equation, but this does not affect whether or not the set has a solution.   We  denote by  $c(\Sigma)$ the subset of $X$ consisting of the variables appearing in equations in $\Sigma$.

\begin{Def}\label{def:solution} Let $\Sigma=\Sigma(X)$ be a set of equations over an $S$-act $A$. A {\em solution} $(b_x)_{x\in X}$ of
$\Sigma(X)$  in $B$ consists of a subset $\{ b^x:x\in X\}$ of $B$,
where $A$ is a subact of $B$,  such that
$b^xs=b^yt$ for all $xs=yt\in \Sigma$ and $b^xs=a$ for all $xs=a\in \Sigma$.
\end{Def}

In the above, if
$X=\{ x_1,\hdots ,x_n\}$ then we may denote
$(b^{x_i}=b_i)_{1\leq i\leq n}$ by  $(b_1,\hdots, b_n)$, and say $(b_1,\hdots, b_n)$ is a {\em solution} of $\Sigma $ or $\Sigma(b_1,\hdots, b_n)$ {\em holds}.
Since we are only interested in when equations have solutions, we freely identify $xs=yt$ with $yt=xs$ and $xs=a$ with $a=xs$.

 The following is essentially a result of universal algebra, but it is convenient to make it explicit. The proof is routine.

\begin{Lem}\label{lem:retract} Let $\Sigma=\Sigma(X)$ be a set of equations over an $S$-act $A$ and let \[
\kappa_{\Sigma}=\{ (xs,yt),\, (zu,a):xs=yt, zu=a\in \Sigma\}.\]
 A solution $(b_x)_{x\in X}$ of $\Sigma$ in $A$ corresponds exactly to a retract $\varphi:A\,\dot{\cup}\, F_S(X)\rightarrow A$ such that
 $\kappa_{\Sigma}\subseteq \ker \varphi$ and $b_x=x\varphi$ for each $x\in X$.
\end{Lem}

Let $\Sigma=\Sigma(X)$ be a set of equations over an $S$-act $A$. If $A$ is a subact of $B$ then we may regard $\Sigma$ as a set of equations over  $B$. As a consequence of  Lemma~\ref{lem:retract} we have the following.

\begin{Lem}\label{lem:oneway} Let $\Sigma$ be a set of equations over $A$, where $A$ is a retract of an $S$-act $B$. If $\Sigma$ has a solution in $B$, then $\Sigma$ has a solution in $A$.
\end{Lem}

We now formally define consistency for a set of equations.

\begin{Def}\label{lem:consistent} A set of equations $\Sigma$ over an $S$-act $A$ is {\em consistent} if it has a solution in some $S$-act $B$ containing $A$.\end{Def}

 We return to the form of equations, to establish the  notions of purity we are concerned with in this  article.

\begin{Def}\label{def:frame} A {\em frame}  (with variables from $X$) is a non-empty set $\FF=\FF_S(X)$ of equation forms. For a frame  $\FF$   we let
\[\FF^2=\FF^2_S(X)=\FF_S(X)\cap (F_S(X)\times F_S(X)\big)
\mbox{ and }\FF^1=\FF^1_S(X)=\FF_S(X)\cap F_S(X).\]
A {\em frame set}  (with variables from $X$) is a set of frames $\F=\F_S(X)$.
\end{Def}

In Definition~\ref{def:frametoeqns} we use the notion of a multimap. If $U$ and $V$ are sets, then by a
{\em multimap} $\phi:U\rightarrow V$ we mean a subset $\phi$ of
$U\times V$, such that the projection onto the first co-ordinate is onto. This notion is chosen for convenience: if $U=\emptyset$, then $\phi=\emptyset$, but if $U\neq\emptyset$, then $u\phi:=\{ v:(u,v)\in \phi\}\neq\emptyset$.

\begin{Def}\label{def:frametoeqns} Let $\FF$ be a frame, let $A$ be an $S$-act and let $\phi:\FF^1\rightarrow A$ be a multimap. Then
\[\Sigma=\Sigma(\FF,\phi)=\{
xs=yt, \, zu=(zu)\phi: (xs,yt)\in \FF^2, \, zu\in \FF^1\}\]
is {\em the set of equations over $A$ with frame $\FF$ and assignment $\phi$. }
\end{Def}

Notice that a frame $\FF$ with $\FF^1\neq\emptyset$ can give rise to different sets of equations, depending on the choice of $A$ and $\phi$.

\begin{Def}\label{def:eqnstoframes} Let $\Sigma$ be a set of equations over an $S$-act $A$. Then the {\em frame
$\FF(\Sigma)$ of $\Sigma$} is defined by
\[\FF(\Sigma)=\{ (xs,yt), zu: xs=yt, zu=a\in \Sigma, a\in A\}.\]
The {\em multimap} $\phi=\phi(\Sigma)$ where $\phi:\FF^1\rightarrow A$ is defined by \[(zu)\phi =a, \mbox{ where }zu=a\in \Sigma.\]
\end{Def}

If $\Sigma$ is a set of equations over an $S$-act $A$  and $\FF=\FF(\Sigma)$ and  $\phi=\phi(\Sigma)$ are defined as above, then
$\Sigma=\Sigma(\FF,\phi)$.
If  $\Sigma=\Sigma(\FF,\phi)$ is consistent, then it can contain at most one equation with equation form $xs$ for any $xs\in F_S(X)$; this corresponds to $\phi$  being a map (with possibly empty domain). Since we are almost always concerned with consistent sets of equations, almost always our multimaps will be maps.

\begin{Def}\label{def:fpure} Let $\F$ be a frame set.
An $S$-act is {\it $\F$-pure} if every  consistent set of equations $\Sigma$ over $A$ with $\FF(\Sigma)\in\F$ has a solution  in $A$.
\end{Def}

There are some important special kinds of frame sets $\F$, resulting in important special kinds of $\F$-purity; we give the examples we need in this article in Definition~\ref{def:purity} below.

\begin{Prop}\label{prop:retractoneway} Let $A$ be an $S$-act. Suppose that $A$ is a retract of an $\F$-pure $S$-act. Then $A$ is $\F$-pure.
\end{Prop}
\begin{proof} Let $B$ be $\F$-pure and let $\varphi:B\rightarrow A$ be a retract.  Let $\Sigma=\Sigma(X)$ be a  consistent set of equations over $A$ with $\FF(\Sigma)\in \F$.  Given that  unions of $S$-acts are $S$-acts, it is easy to see that
$\Sigma$ may be regarded as a consistent set of equations over $B$, so has a solution $(b_x)_{x\in X}$ in  $B$. Since $A$ is a retract of $B$, $(b_x\varphi)_{x\in X}$  is a solution of $\Sigma$ in $A$. Hence $A$ is $\F$-pure.
\end{proof}

Much of what we do is to build towards a converse of Proposition ~\ref{prop:retractoneway} - for this we need to construct specific extensions of $A$ of which $A$ is a retract.
We are interested in conditions on $A$ such that a given set $\Sigma$ of equations  has a solution in $A$. We remark that it is irrelevant how the variables of
such a $\Sigma$ are labelled; for example, $(b_1,\hdots ,b_n)$ is a solution of
$\Sigma(x_1,\hdots, x_n)$ if and only if it is a solution of
$\Sigma(y_1,\hdots ,y_n)$.  To prevent complete explosion of notational complexity, we  may change the labelling of the variables in a set $\Sigma$
without comment.

One reason why equations over $S$-acts are amenable to study is that we have a criterion for consistency of a set of equations: this is such that, if $S$  is finite, then it is decidable whether a set of equations is consistent. We now outline the relevant ideas, which will be useful throughout this article.

To any frame  $\FF=\FF_S(X)$
we let
\[H(\FF)=\FF^2, \rho_{\mathcal{F}}=\langle H(\FF)\rangle, \,
C(\FF)=F_S(X)/\rho_{\mathcal{F}}\mbox{ and }B(\FF)=\cup_{xs\in \mathcal{F}^1}[xs]S.\]
If $\FF^1=\emptyset$ then $B(\FF)=\emptyset$.
Correspondingly, if $\Sigma=\Sigma(X)=\Sigma(\FF,\phi)$ is a set of equations over $A$  we let
\[ H(\Sigma)=H(\FF), \, \rho_{\Sigma}=\rho_{\mathcal{F}},\,
C(\Sigma)=C(\FF)\mbox{ and }B(\Sigma)=B(\FF).\]
In addition,  we define
\[K(\Sigma)=\{(xs,a):xs=a\in \Sigma\},\]
so that the  congruence $\kappa_{\Sigma}$  on $A\,\dot{\cup}\, F_S(X)$ may be defined by
\[
\kappa_{\Sigma}=\langle H(\Sigma)\cup K(\Sigma)\rangle.\]

\noindent Continuing, we let
\[
\tau_{\Sigma}:A\,\dot{\cup}\, F_S(X)\rightarrow A(\Sigma)\]  be the natural map,  with restriction denoted by
\[\nu_{\Sigma}=\tau_{\Sigma}|_{A}:A\rightarrow A(\Sigma) \]  so that  $a\tau_{\Sigma}=a\nu_{\Sigma}=[a]$. The set of equations
 which we obtain from $\Sigma $ by replacing each equation of the form $xs=a$ by $xs=[a]$ has a solution in $A(\Sigma)$.
Finally, we let
\[\theta_{\Sigma}:B(\Sigma)\rightarrow A\]
be defined by
\[([xs]u)\theta_{\Sigma}=au,\mbox{ where }a=(xs)\phi,
\mbox{ that is, }xs=a\in \Sigma.\]
Notice that at this stage we are not claiming that $\theta_{\Sigma}$ is well defined.

The following three propositions, which we use frequently in our arguments,   are implicit in \cite[Lemma 2.3]{gould:1986}, although not always stated there in full.  For completeness we  state the results in the form required here and provide outline  proofs.

\begin{Prop}\label{prop:easy} Let $\Sigma(X)$ be a  consistent set of equations over an $S$-act $A$ with solution $(b_y)_{y\in X}$. Then for all $ys,zt\in F_S(X)$ we have
\[ys\,\rho_\Sigma \, zt\Rightarrow b_ys=b_zt.\]
\end{Prop}
\begin{proof}
Suppose that $ys~\rho_{\Sigma}~zt$.  There exists  an $H(\Sigma)$-sequence
$$ys=c_1t_1, d_1t_1=c_2t_2, \cdots, d_nt_n=zt$$ where $n\in \N^0$, $t_i\in S$ and $(c_i, d_i)\in H(\Sigma)$ for all $1\leq i\leq n$.
Notice that the equalities are in the free $S$-act $F_S(X)$.  If $n=0$ then $ys=zt$ so that
$y=z, s=t$ and $b_ys=b_zt$. If $n\geq 1$ then we  have $(c_1,d_1)=(yh,wk)\in H(\Sigma)$ so that $yh=wk$ is an equation in $\Sigma$.
Then as $s=ht_1$ we have $b_ys=b_yht_1=b_wkt_1$ and
\[w(kt_1)=c_2t_2,\cdots, d_nt_n=zt\]
is an $H(\Sigma)$-sequence of length $n-1$ joining $w(kt_1)$ to $zt$.
Induction now yields the result.
\end{proof}

\begin{Prop}\label{prop:soln} Let $\Sigma=\Sigma(X)$
be a set of equations over $A$. Then the following conditions are equivalent:
\begin{enumerate} \item $\Sigma$ is consistent;
\item for all $xs=a, yt=b\in \Sigma$ and $v,w\in S$,
\[ xsv \,\rho_\Sigma\, ytw
\Rightarrow av=bw;\]
\item  $\theta_{\Sigma}:B(\Sigma)\rightarrow A$ is well-defined (and is an $S$-morphism);
\item  $\nu_\Sigma$ is an embedding of $A$ into $A(\Sigma)$.
\end{enumerate}
If any of these conditions hold, then $([x])_{x\in X}$ is a solution of $\Sigma$ in $A(\Sigma)$.
\end{Prop}

\begin{proof} Suppose that (1) holds and $(b_x)_{x\in X}$ is a solution of $\Sigma$. If $xs=a, yt=b\in \Sigma$ with $xsv \,\rho_\Sigma\, ytw$, then from Proposition~\ref{prop:easy}
we have \[av=b_xsv=b_ytw=bw,\]giving that (2) holds.

Suppose that (2) holds and $[a]=[b]$ for $a,b\in A$: we  show that $a=b$. We either have this immediately, or else an $H(\Sigma)\cup K(\Sigma)$-sequence
\[a=\alpha_1 t_1, \beta_1t_1=\alpha_2 t_2, \cdots, \beta_n t_n=b\] where $t_i\in S$ and $(\alpha_i, \beta_i)\in H(\Sigma)\cup K(\Sigma)$ for all $1\leq i\leq n$. Here we must have $\alpha_1, \beta_n\in A$ and so $(\alpha_1,\beta_1)=(c,xs)$
and $(\alpha_n,\beta_n)=(yt,d)$ where
$c=xs, yt=d\in\Sigma$. Assume that $(\alpha_i, \beta_i)\in H(\Sigma)$ for all $2\leq i\leq n-1$. We then have that $\beta_1t_1=xst_1$ and $\alpha_nt_n=ytt_n$. By (2) we have that $ct_1=dt_n$ and so
$a=ct_1=dt_n=b$. Induction allows us to conclude that (4) holds.

If (4) holds then, given an earlier remark,  identifying $A\kappa_{\Sigma}$ with $A$ yields (1).  Finally, (2) and (3) are essentially reformulations of each other.
\end{proof}

 Notice that, in the above, if $B=\emptyset$, which corresponds to there being no equations with constants, or equivalently $\FF^1=\emptyset$, then any such set of equations is consistent. Indeed, any such set has a solution $o$ in $A\, \dot{\cup}\, \{ o\}$ where
$\{ o\}$ is a trivial (one-element) $S$-act.

If $\Sigma$ is consistent, then in general, as above, it is convenient to identify $A$ with $A\nu_{\Sigma}$.

\begin{Prop}\label{cor:ret} Let $\Sigma=\Sigma(X)$
be a consistent set of equations over $A$. Then the following conditions are equivalent:
\begin{enumerate} \item $\Sigma$ has a solution in $A$;
\item $A$ is a retract of $A(\Sigma)$;
\item the $S$-morphism $\theta_{\Sigma}:B(\Sigma)\rightarrow A$ lifts to an $S$-morphism $\overline{\theta_{\Sigma}}:C(\Sigma)\rightarrow A$.
\end{enumerate}
\end{Prop}
\begin{proof} If $\Sigma$ has a solution in $A$, then by Lemma~\ref{lem:retract} there is a retraction $\varphi:
A\, \dot{\cup}\, F_S(X)\rightarrow A$ such that
 $\kappa_{\Sigma}\subseteq\ker\varphi$. We may now define an $S$-morphism  $\overline{\varphi}:
A(\Sigma)\rightarrow A$ by $[t]\overline{\varphi}=t\varphi$ which, since $\Sigma$ is consistent, is a   retraction
by Proposition~\ref{prop:soln}.

Conversely, if $A$ is a retract of $A(\Sigma)$ then as $\Sigma(X)$ has a solution in $A(X)$ it must have a solution in $A$. Therefore, (1) and (2) are equivalent.

To show (1) implies (3), we define a map $\theta_{\Sigma}': F_S(X)\rightarrow A$ by $y\theta_{\Sigma}'=b_y$ where $(b_y)_{y\in X}$ is a solution of $\Sigma$ in $A$. Clearly, $H(\Sigma)\subseteq \ker \theta_{\Sigma}'$, and so $\overline{\theta_{\Sigma}}: C(\Sigma)\rightarrow A$ defined by $[t]\overline{\theta_{\Sigma}}=t\theta_{\Sigma}'$ is a well-defined morphism. Further, it is easy to check that $\overline{\theta_{\Sigma}}|_{B(\Sigma)}=\theta_{\Sigma}$. Conversely, suppose that (3) holds. Then $([y]\overline{\theta_{\Sigma}})|_{y\in X}$ is a solution of $\Sigma$ in $A$.
 Therefore, (1) and (3) are equivalent.
\end{proof}

We now give the promised connections between $\F$-purity and weak injectivity properties.

\begin{Thm}\label{thm:connection} Let $\F$ be a frame set and let $A$ be an $S$-act. Then $A$ is $\F$-pure if and only if every diagram of the form on the left, where $\FF\in \F$  and $\theta$ is an $S$-morphism,
\begin{center}

\begin{tikzpicture}

\node at (0,0) {$C(\FF)$};
\node at (3,0) {$B(\FF)$};
\node at (3,2) {$A$};

\node at (3.2,1) {$\theta$};



\draw [left hook-> , thick] (2.5,0) -- (.5,0);

\draw [->, thick] (3,0.3) -- (3,1.7);


\node at (6,0) {$C(\FF)$};
\node at (9,0) {$B(\FF)$};
\node at (9,2) {$A$};

\draw [left hook-> , thick] (8.5,0) -- (6.5,0);
\node at (9.2,1) {$\theta$};

\node at (7.4,1.5) {$\overline{\theta}$};

\draw [->, thick] (9,0.3) -- (9,1.7);

\draw [->, thick] (6.2,0.3) -- (8.8,1.7);

\end{tikzpicture}\end{center}  can be completed as in the diagram on the right,
where $\overline{\theta}$ is an $S$-morphism.

\end{Thm}
\begin{proof}  Suppose first that $A$ is $\F$-pure and
$\FF\in \F$ is such that $\theta$ exists as given. For $xs\in \FF^1$ we have $[xs]\in B(\FF)$; put  $a_{xs}=[xs]\theta$ and $(xs)\phi=[xs]\theta$.   Now let $\Sigma=\Sigma(\FF,\phi)$. Then $\theta=\theta_{\Sigma}$ is certainly well-defined, so by Proposition~\ref{prop:soln} we have that $\Sigma$ is consistent. By assumption, $\Sigma$ has a solution $(b_x)_{x\in X}$ in $A$.  By the proof of  Proposition~\ref{cor:ret},  $\overline{\theta}=\overline{\theta_{\Sigma}}: C(\FF)\rightarrow A$ given by $[xs]\overline{\theta}=b_xs$  is a well-defined $S$-morphism extending $\theta$.

Conversely, suppose that any diagram of the given form can be completed. Let $\Sigma=\Sigma(\FF,\phi)$ be a consistent set of equations over $A$ with $\FF\in \F$ and let $\theta=\theta_{\Sigma}:B(\FF)\rightarrow A$. By Proposition~\ref{prop:soln}, $\theta$ is a well-defined $S$-morphism. By assumption, $\theta:B(\Sigma)\rightarrow A$ lifts to an $S$-morphism $\overline{\theta}:C(\Sigma)\rightarrow A$. The result now follows from Proposition~\ref{cor:ret}.
\end{proof}

In the above, where $B(\FF)=\emptyset$, that is, $\FF^1=\emptyset$, completion of the diagram is interpreted as meaning the existence of a morphism $C(\FF)\rightarrow A$.

We now define the various special frames and frame sets in which we will be interested.

\begin{Def}\label{def:purity}
\begin{enumerate} \item A frame $\FF=\FF_S(X)$ is an {\em fp-frame} if $\FF$ is finite and $B(\FF)$ has a finite presentation. If  $\F$ is the frame set
of all fp-frames,   then we refer to an $\F$-pure act as being {\em fp-pure}.
\item A frame $\FF=\FF_S(X)$ is an {\em mfp-frame} if $|X|=1$ and it is an fp-frame.  If  $\F$ is the frame set
of all mfp-frames,   then we refer to an $\F$-pure act as being {\em mfp-pure}.
\item  A frame $\FF=\FF_S(X)$ is an {\em $n$-frame} if $\FF$ is finite and $|X|\leq n$. If  $\F$ is the frame set
of all $n$-frames,   then we refer to an $\F$-pure act as being {\em $n$-absolutely pure}.
\item If $\F$ is the frame set
of all $1$-frames over $X$, then we refer to an $\F$-pure act as being {\em almost pure}.
\item If $\F$ is the frame set
of all finite frames over $X$, then we refer to an $\F$-pure act as being {\em absolutely pure}.
\end{enumerate}\end{Def}

Applying Theorem~\ref{thm:connection} to the frame sets in Definition~\ref{def:purity}  we have the following, which was known in the case of (4) and (5) \cite[Proposition 3.8]{gould:1985}.

\begin{cor}\label{cor:thetransfer} Let $A$ be an $S$-act. Then
\begin{enumerate}
\item $A$ is fp-pure if and only if it is injective with respect to inclusions of finitely presented subacts of finitely presented  $S$-acts;
\item $A$ is mfp-pure if and only if it is injective with respect to inclusions of finitely presented subacts of finitely presented monogenic $S$-acts;
\item $A$ is $n$-absolutely pure if and only if it is injective with respect to inclusions of finitely generated subacts of finitely presented  $S$-acts having no more than $n$ generators;
\item $A$ is almost pure if and only if it is injective with respect to inclusions of finitely generated subacts of finitely presented monogenic $S$-acts;
\item $A$ is absolutely pure if and only if it is injective with respect to inclusions of finitely generated subacts of finitely presented  $S$-acts.
\end{enumerate}
\end{cor}

Considering the frame set of {\em all} frames we immediately have:

\begin{cor}\label{cor:inj}\cite[Proposition 3.10]{gould:1985}   Let $A$ be an $S$-act. Then $A$ is injective if and only if every consistent set of equations over $A$ has a solution in $A$.
\end{cor}

\begin{Def}\label{def:theclasses}
 We denote by $\mathcal{A}^{fp}_S(1)$, $\mathcal{A}_S(1)$ and  $\mathcal{A}_S(\aleph_0)$ the classes of mfp-pure, almost pure and
absolutely pure $S$-acts, respectively.
\end{Def}

Our terminology, referring to {\em purity}, comes from the completion of diagrams.
Alternative terminology, focussing on the  equations, is {\em $n$-algebraically closed} (for $n$-absolutely pure) and {\em algebraically closed}  (for absolutely pure).

Sets of equations without any constants are rather special. In this regard we need the following definition.

\begin{Def}\label{def:local} Let $A$ be an $S$-act. Then $A$ has
{\em local left zeros} if for any finite set $T\subseteq S$ there is a $a=a_T\in A$ such that $a=at$ for each $t\in T$.
\end{Def}

Clearly, if $A$ has local left zeros, then any finite set of equations without constants is consistent over $A$ and indeed has a solution in $A$. For a converse we have the following, which can be extracted from earlier works, for example \cite{gould:1987}, but which for convenience we prove explicitly.

\begin{Prop}\label{prop:local} Let $A$ be an $\F$-pure $S$-act where $\F$ contains all finite frames in one variable contained in $F_S(X)\times F_S(X)$. Then $A$ has local left zeros.
\end{Prop}
\begin{proof} Let $T\subseteq S$ be finite and consider the set of equations $\Sigma(x)=\{ x=xt:t\in T\}$. As remarked earlier, $\Sigma(x)$ has a solution in $A\,\dot{\cup}\, \{ o\}$. Since $A$ is $\F$-pure and $\FF(\Sigma)\in \F$, we have that $\Sigma(x)$ has a solution, say $a\in A$. Clearly $a=at$ for each $t\in T$.
\end{proof}

\section{Purity of $S$-acts over right coherent monoids }\label{sec:coherent}

The aim of this section is to show that  for any right coherent monoid $S$  all almost pure $S$-acts must be absolutely pure, that is, $\mathcal{A}_S(\aleph_0)=\mathcal{A}_S(1)$. The very fact that $S$ is right coherent, then yields that for such $S$ it follows that $\mathcal{A}_S(\aleph_0)=\mathcal{A}_S(1)=\mathcal{A}^{fp}_S(1)$. As finite monoids are right coherent, we deduce that the condition that $\mathcal{A}_S(1)=\mathcal{A}_S(\aleph_0)$ is  a finitary property  for monoids.

\begin{them}\label{coherent monoids} Let $S$ be a right coherent monoid. Then an $S$-act $A$ is almost pure if and only if it is absolutely pure.

\end{them} \begin{proof} Let $\Sigma=\Sigma(X)$ be a finite consistent set of equations over $A$. If $|X|=1$, then, as $A$ is almost pure, $\Sigma$ has a solution in $A$. Proceeding by induction, we suppose that $|X|=n\geq 2$ and every finite consistent set of equations over $A$ in at most $n-1$ variables has a solution in $A$.

From Proposition~\ref{prop:local} $A$ has local left zeros. Thus, if $\Sigma$ contains no equations with constants, we can construct a solution to  $\Sigma$  in $A$, as commented before that proposition.

Suppose therefore that $\Sigma$ contains at least one equation with a constant; suppose that the variable for that equation is $x$. Let $(b_y)_{y\in X}$ be a solution for $\Sigma$ and for ease let $b_x=b$.

Let $F_S(X)$ be the free $S$-act on $X$ and let
$\rho_\Sigma$ be defined as in Section~\ref{sec:pre}.

  We use Theorem~\ref{thm:coherency} to build a new consistent set of equations $\Pi(x)$ in the single variable $x$.

{\em Step (a)}  For each $xs\in F_S(X)$, consider $$\bold{r}([xs])=\{(u, v)\in S\times S: xsu~\rho_{\Sigma}~xsv\}.$$ Since $S$ is right coherent,
Theorem~\ref{thm:coherency} gives that $\bold{r}([xs])$ is a finitely generated right congruence on $S$.  We use $H(xs)$ to denote a fixed finite generating set of $\bold{r}([xs])$.   Notice that for all $(u,v)\in H(xs),$ or more generally, $(u, v)\in \bold{r}([xs])$, we have $xsu~\rho_\Sigma~xsv$ and so, by Proposition~ \ref{prop:easy}, $bsu=bsv.$

{\em Step (b)} For each pair of equations $xs=yt, zu=d\in \Sigma(X)$ with $y\neq x$ such that $[xs]S\cap [zu]S\neq \emptyset,$ then, again as $S$ is right coherent, Theorem~\ref{thm:coherency} yields $[xs]S\cap [zu]S$ is finitely generated as a subact of $F_S(X)/\rho_\Sigma$. Let $K=K(xs=yt, zu=d)$ denote a fixed finite subset of $S$ such that $$[xs]S\cap [zu]S=\underset{{k\in {K}}}\cup[x]k S.$$ For each $k\in K$, we use $k_{xs}$ and $k_{zu}$ to denote some fixed elements in $S$ such that $[xk]=[xsk_{xs}]=[zuk_{zu}].$  Then we have $xk~\rho_\Sigma~xs k _{xs}~\rho_\Sigma~zuk_{zu}$, so that $bk=bs k_{xs}=b_zu k_{zu}$ by Proposition \ref{prop:easy}. Notice that $b_zu=d\in A$, so that certainly $b_zu k_{zu}\in A$.

{\em Step (c)} For each pair of equations $xs =yt, xu=zv \in \Sigma(X)$ with $y,z\neq x$ such that $[xs ]S\cap [xu]S\neq \emptyset,$ let $L=L(xs =yt, xu=zv)$ be a fixed finite subset of $S$ such that
\[[xs]S\cap [xu]S=\underset{l\in L}\cup [x]l S.\]
 For each $l\in L$, let $l_{xs}, l_{xu}\in S$ be fixed elements in $S$ such that $[xl]=[xs l_{xs}]=[xu l_{xu}].$ Then  $xl~\rho_\Sigma~xs l_{xs}~\rho_\Sigma~xu l_{xu}$ and so $bl=bs l_{xs}=bul_{xu}$  by Proposition \ref{prop:easy}.

\smallskip

Let $\Sigma(x)$ be the set of all equations of $\Sigma(X)$ in which the only variable that occurs is $x$. Define $$\Pi(x)=\Sigma(x)\cup \Sigma_1(x)\cup \Sigma_2(x)\cup \Sigma_3(x),$$ where
 \[ \begin{array}{lll} \begin{array}{rclrcl}\Sigma_1(x)=\{xsu=xsv:  xs=yt\in \Sigma(X), (u, v)\in H(xs), y\neq x\},\end{array}\\
         \begin{array}{rclrcl}\Sigma_2(x)=\{xsk_{xs}=b_zuk_{zu}:  xs=yt, zu=d\in \Sigma(X), [xs]S\cap[zu]S\neq \emptyset,\\
  y, z\neq x, k\in K(xs=yt, zu=d)
\}\end{array},\\
        \begin{array}{rclrcl}\Sigma_3(x)=\{xsl_{xs}=xul_{xu}:  xs=yt, xu=zv\in \Sigma(X), [xs]S\cap[xu]S\neq \emptyset,\\
  y, z\neq x,l\in L(xs=yt, xu=zv)
\}\end{array}.\end{array} \]

It follows from the above Steps (a), (b) and (c) that $\Pi(x)$ is a finite consistent set of equations with a solution $b$. As $A$ is almost pure, $\Pi(x)$   has a solution $c$ in $A$. Notice that for $xs =yt \in \Sigma(X)$ with $x\neq y$ and $(u, v)\in H(xs)$, we have $csu=csv$ by the construction of $\Sigma_1(x)$. Let $(g, h)\in \mathbf{r}([xs])$. Then $g=h$, so that $csg=csh$, or there exists an  $H(xs)$-sequence
 $$g=u_1t_1, v_1t_1=u_2t_2, \cdots, v_mt_m=h$$ where $(u_i, v_i)\in H(xs)$ and $t_i\in S$ for all $1\leq i\leq m$. In this latter case,  $csu_i=csv_i$ for all $1\leq i\leq m$, giving $$csg=cs u_1t_1=cs v_1t_1=csu_2t_2=\cdots=csv_mt_m=csh.$$

Now let $\Sigma'(x)$ be the set of all equations of $\Sigma(X)$ in which $x$ appears, so that $$\Sigma'(x)=\Sigma(x)\cup \{ xs=yt: xs=yt \in \Sigma(X), x\neq y\}.$$ Let $Y=X\setminus \{ x\}$. Define
\[\overline{\Sigma}=\Sigma(Y)=\big(\Sigma(X)\setminus \Sigma'(x)\big)\cup
\{ cs=yt: xs=yt\in \Sigma(X), y\neq x\}.\]
We claim that $\overline{\Sigma}$ is consistent. To this end, let $F_S(Y)$ be the free
$S$-act on $Y$.  Then
\[\rho_{\overline{\Sigma}}=\langle(yt,zu):yt=zu\in \Sigma(Y)\rangle\subseteq \rho_\Sigma.\]

Let $yt=a, zu=d\in \Sigma(Y)$ with $ytg \,\rho_{\overline{\Sigma}}\, zuh$ for some $g, h\in S$.
We must show that $ag=dh$. We consider the following three cases.

\bigskip{\em Case (i)} $yt=a, zu=d\in \Sigma(X)$ with $y, z\neq x$. Then $ag=dh$
by the consistency of $\Sigma(X)$.

\bigskip{\em Case (ii)}  $yt=xs,  zu=d\in \Sigma(X)$ with $ y, z\neq x,  a=cs$.
We have
\[xsg\,\rho_{{\Sigma}}\,ytg\,\rho_{{\Sigma}}\, zuh \] so that $bsg=b_ytg=b_zuh$ and also  $[xs]S\cap [zu]S\neq \emptyset.$ Then for all $k\in K$ we have $xs k_{xs}=b_zuk_{zu}\in \Sigma_2(x)$ and so $$csk_{xs}=b_zuk_{zu}.$$ Further, since $$[zuh]\in [xs]S\cap [zu]S=\underset{k\in K}\cup [xk] S=\underset{k\in K}\cup [xsk_{xs}]S=\underset{k\in K}\cup [zuk_{zu}]S$$ there exists  $k\in K$
 and $p \in S$  such that $$zuh~\rho_{\Sigma}~xkp~\rho_{\Sigma}~xs k_{xs}p~\rho_{\Sigma}~zuk_{zu}p,$$ giving
 $xsg~\rho_{\Sigma}~xsk_{xs}p,$ and so $(g, k_{xs}p)\in  \bold{r}([xs])$. Now we have
 $$ag=cs g=cs k_{xs}p=b_zuk_{zu}p=b_zuh=dh.$$

\bigskip{\em Case (iii)}  $xs=yt, xv=zu\in \Sigma(X)$ with $y, z\neq x, a=cs, d=cv$.
We have
\[xsg~\rho_{\Sigma}~ytg~\rho_{\Sigma}~zuh~\rho_{\Sigma}~xvh\]
giving $[xs]S\cap[xv]S\neq \emptyset.$ Then for all $l\in L$ we have  $xsl_{xs}=xvl_{xv}\in \Sigma_3(x),$ and so $$cs l_{xs}=cvl_{xv}.$$
Further, since $$[xsg]\in [xs]S\cap[xv]S=\underset{l\in L}\cup [x]l S=\underset{l\in L}\cup [xs l_{xs}] S=\underset{l\in L}\cup [xvl_{xv}]S$$
there exists  $l\in L$ and  $q\in S$ such that $$xvh~\rho_{\Sigma}~xsg~\rho_{\Sigma}~xlq~\rho_{\Sigma}~xsl_{xs}q~\rho_{\Sigma}~xvl_{xv}q.$$ Notice that $(h, l_{xv}q)\in \mathbf{r}([(xv)]$ and $(g, \l_{xs}q)\in \mathbf{r}([(xs)]$, so we have  $$ag=csg=csl_{xs}q=cv
l_{xv}q=cvh=dh.$$

Therefore we have  that $\Sigma(Y)$ is a finite consistent set of equations in
$|Y|=n-1$ variables over $A$, so by our inductive hypothesis, $\Sigma(Y)$ has a solution
$(c_y)_{y\in Y}$ in $A$. Putting $c_x=c$ it is easy to see that
$(c_y)_{y\in X}$ is a solution to $\Sigma(X)$. This completes the proof.
\end{proof}

 As shown in Theorem \ref{fountain monoid}, the converse of Theorem \ref{coherent monoids} is not true, in general.

The next corollary confirms that $\mathcal{A}_S(1)=\mathcal{A}_S(\aleph_0)$ is indeed a finitary  property for monoids. It follows from the fact that right coherency is a finitary property, and Theorem \ref{coherent monoids}.

\begin{Cor}\label{cor:ap=alp}
Let $S$ be a finite monoid. Then every almost pure $S$-act is absolutely pure.
\end{Cor}

\section{Canonical constructions}\label{sec:const}

It is clear from Theorem~\ref{coherent monoids} and its proof that  right coherency of $S$ is strongly related to the property that  $\mathcal{A}_S(1)=\mathcal{A}_S(\aleph_0)$. The main results of the remaining sections,
Theorem~\ref{mfp:main} and Theorem~\ref{main theorem}, add  to this evidence. The purpose of the current section is to provide the machinery to prove these theorems.
Building on techniques established in Section~\ref{sec:pre}, for any  frame set  $\F$, we construct a canonical $\F$-pure extension $A(\F)$ of an arbitrary $S$-act $A$.  Where $\F$ is  the set of all mfp-frames (1-frames,  finite frames) then we denote $A(\F)$ by $A(1)^{fp}$ ($A(1)$,  $A(\aleph_0)$), so that these are
  canonical mfp-pure (almost pure,  absolutely pure) extensions  of $A$.  In Section~\ref{sec:mfp}  we use $A(1)^{fp}$ to prove Theorem~\ref{mfp:main}, which states that  all mfp-pure acts are almost pure,  that is,
 $\mathcal{A}_S^{fp}(1)=\mathcal{A}_S(1)$, if and only if $S$ is right coherent. In Section~\ref{sec:ap=ap} we explicitly use  $A(1)$ and $A(\aleph_0)$  to establish Theorem~\ref{main theorem}, which gives conditions for all almost pure $S$-acts to be absolutely pure, that is,
 $\mathcal{A}_S(1)=\mathcal{A}_S(\aleph_0)$, in terms of finitely presented $S$-acts, their finitely generated $S$-subacts and their canonical extensions.

The $S$-acts that we build are constructed from  infinite towers of extensions of $A$: strictly speaking we cannot merely take their union as we do not have a universal $S$-act of which they are all subacts. Rather, we are taking a direct limit where we are suppressing explicit notation for the embedding of one subact into another.

In what follows, it is convenient to say that $\Sigma$ is {\em a set of $\F$-equations} if $\FF(\Sigma)\in \F$.

\begin{Def}\label{def:built} Let $A$ be a subact of an $S$-act $B$. We say $B$ is {\em $\F$-built}  from $A=A_0$ if for some ordinal $\xi$ we have
\[B=\bigcup_{0\leq i\leq \xi}  A_i\]
where:

(i) for each $0\leq i<\xi$, the subact $A_{i+1}=A_i(\Sigma_i)$ for some consistent set $\Sigma_i$ of $\F$-equations  over $A_i$;

(ii) if $\zeta$ is a limit ordinal, then $A_{\zeta}=\bigcup_{0\leq i<\zeta}A_i$.
\end{Def}

For our next result we require a pair of  technical lemmas.

\begin{Lem}\label{lem:buildingstart} Let $A$ be a subact of an $S$-act $B$. Suppose that $\theta:B\rightarrow A$ is an $S$-morphism. Let $\Sigma=\Sigma(X)$ be a consistent set of $\F$-equations over $B$. Then $\Sigma_{\theta}$, where $\Sigma_{\theta}$ is obtained from $\Sigma$ by replacing each constant $c$ by $c\theta$, is consistent over $A$ and is  a set of $\F$-equations.  Further, $\bar{\theta}:
B(\Sigma)\rightarrow A(\Sigma_{\theta})$ given by
\[[x]\bar{\theta}=[x]\mbox{ and }b\bar{\theta}=b\theta,\]
for $x\in X$ and $b\in B$ (with appropriate interpretation of equivalence classes)
is an $S$-morphism extending $\theta$.
\end{Lem}
\begin{proof} By Proposition~\ref{prop:soln} the set $\Sigma_{\theta}$ is consistent; it follows from the definition that if $\Sigma$ has frame  in $\F$, then so does  $\Sigma_{\theta}$. Again from their definitions, with an application of the first isomorphism theorem,  it is easy to see that there is an $S$-morphism
$\bar{\theta}: B(\Sigma)\rightarrow A(\Sigma_{\theta})$ with the required properties.
\end{proof}

\begin{Lem}\label{lem:building} Let $A$ be   a subact of an $S$-act $B$. Suppose that $\theta:B\rightarrow A$ is an $S$-morphism. Let $\Sigma=\Sigma(X)$ be a consistent set of $\F$-equations over $B$. Then
if $A$ is $\F$-pure there is an $S$-morphism from $B(\Sigma)$ to $A$ extending $\theta$  which is a retract if $\theta $ is a retract.
\end{Lem}
\begin{proof} Following the notation and conclusion of Lemma~\ref{lem:buildingstart} we have an $S$-morphism
$\bar{\theta}:
B(\Sigma)\rightarrow A(\Sigma_{\theta})$ such that
$[x]\bar{\theta}=[x]$ and $b\bar{\theta}=b\theta$. Since $A$ is $\F$-pure, Proposition~\ref{cor:ret}  says there is a retract $\psi:A(\Sigma_{\theta})\rightarrow A$, so that certainly $\bar{\theta}\psi:B(\Sigma)\rightarrow A$ is an $S$-morphism extending $\theta$. The final statement is then clear.
\end{proof}

\begin{Prop}\label{prop:built}
 Let $A$ be an $\F$-pure $S$-act, and let $B$ be $\F$-built from $A$. Then $A$ is a retract of $B$.

\end{Prop}
\begin{proof} We show by transfinite induction that for each $0\leq i\leq\xi$ there is a retraction $\varphi_i:A_i\rightarrow A$,  such that for $i<j$ we have $\varphi_j|_{A_i}=\varphi_i$. This is clearly true for $i=0$.

Suppose that $\varphi_j$ has been defined with the required property for all $0\leq j<\mu$. If
 $\mu$ is a limit ordinal we simply define $b\varphi_{\mu}=b\varphi_i$ where $b\in A_i$ and $0<i<\mu$.  On the other hand, if $\mu=i+1$ then we have that $A_{i+1}=A_i(\Sigma_i)$ for some consistent set $\Sigma_i$ of $\F$-equations over $A_i$.
 We  apply Lemma~\ref{lem:building} to construct the required $\varphi_{i+1}$.

 It is  immediate that $\varphi:B\rightarrow A$ given by $b\varphi=b\varphi_i$, where $b\in A_i$, is a retraction.
\end{proof}

We now proceed to build the promised canonical constructions. They are essentially based on the standard way to build an algebraically or existentially closed structure extending a given one, in any class closed under unions of chains. However, to use our constructions to extract results, a little care is required.

For any set of equations $\Sigma=\Sigma(X)$, and any set $Y_X=\{ y_x:x\in X\}$  of new symbols,  we have another set of equations $\Sigma(Y_X)$, with precisely the same consistency properties as the original. Our convention in what follows is that  for any consistent set of equations $\Sigma$ we {\em choose and fix a set of variables}, such that for any  two different sets of equations, we {\em  choose different variables}. The result of this is that if $\Sigma_{i\in I}$ is a set of  consistent sets of equations over $A$, then $\bigcup_{i\in I}A(\Sigma_i)$ is an $S$-act, and for $i\neq j$ we have
$A(\Sigma_i)\cap A(\Sigma_j)=A$; in other words, we can amalgamate $\{A(\Sigma_i):i\in I\} $   over $A$. Here, as elsewhere, we freely identify the image of $A$ in $A(\Sigma)$ with $A$.

Let $A$ be an $S$-act and let $\F$ be a set of frames. Define
\[\Theta(A,\F)=\{ \Sigma:\FF(\Sigma)\in \F, \, \Sigma\mbox{ is  consistent over }A\}\]
and then put
\[\Omega(A,\F)=\bigcup_{\Sigma\in\Theta(A,\mathscr{F})
    }c(\Sigma).\]

Now let
\[A_1^{\mathscr{F}}=(A\,\dot{\cup}\, F_S(\Omega(A,\F))/\kappa(A,\F)\]
where
\[\kappa(A,\F)=\langle H(\Sigma)\cup K(\Sigma):\Sigma\in \Theta(A,\F)\rangle.\]

The next result relies on a remark above, namely  that, due to our labelling of variables,  for distinct $\Sigma,\Sigma'\in\Theta(A,\F)$ we have
$A(\Sigma)\cap A(\Sigma')=A$.

\begin{Lem}\label{lem:firststep} Let $A$ be an $S$-act and let $\F$ and $\mathscr{G}$ be  frame sets with $\F\subseteq \mathscr{G}$. Then
\begin{enumerate}\item The $S$-act
$A_1^{\mathscr{F}}$ is the amalgamation of the S-acts $A(\Sigma)$ where
$ \Sigma\in \Theta(A,\F)$ over $A$, in particular, $A$ is embedded in $A_1^{\mathscr{F}}$;

\item  $A^{\mathscr{F}}_1\subseteq A^{\mathscr{G}}_1$;
\item every consistent set of $\F$-equations  over $A$ has a solution in
$A_1^{\mathscr{F}}$ and hence in
$A_1^{\mathscr{G}}$;
\item $A$ is $\F$-pure if and only if it is a retract of
$A_1^{\mathscr{F}}$.
\end{enumerate}
\end{Lem}
\begin{proof} (1)-(3) are clear,  given our careful labelling of variables in sets of equations; (4) follows from Proposition~\ref{cor:ret}.
\end{proof}

We cannot say, for example, that if $\F$ is the frame set of all finite frames, then $A^{\mathscr{F}}_1$ is absolutely pure, since we have not considered consistent sets of equations with constants in
 $A^{\mathscr{F}}_1\setminus A$. We need to iterate our construction to achieve the desired canonical extensions of $A$.  Figure~\ref{fig:building} gives an illustration.

\begin{figure}[!ht]
\begin{tikzpicture}

  \node at (1.5,2) [right] {$A=A_0^{\mathscr{F}}$};
  \node[color=blue] at (3.3,1) [right] {$A_1^{\mathscr{F}}$};


  \node[color=red] at (5.25,2) [right] {$A_2^{\mathscr{F}}$};
  \node[color=black] at (5.9,2) [right] {...};
  \node[color=green] at (6.70,2) [right] {$A_n^{\mathscr{F}}$};
  \node[color=black] at (7.6,2) [right] {...};
  \node[color=orange] at (8.1,2) [right] {$A(\F)$};

\draw(2,2) ellipse (2 and 1);

\draw[color=blue,very thick,solid] (2.5,2.99) parabola bend (3,3.8) (3,3.8);
\draw[color=blue,very thick,solid] (3,3.8) parabola bend (3,3.8) (3.1,2.83);

\draw[color=blue,very thick,solid] (3.1,2.83) parabola bend (4,3.65) (4,3.65);
\draw[color=blue,very thick,solid] (4,3.65) parabola bend (3.65,2.55) (3.65,2.55);

\draw[color=red,very thick,solid] (3,3.8) parabola bend (3.35,4.1) (3.7,3.6);

\draw[color=blue,very thick,solid] (3.65,2.55) parabola bend (4.9,3) (4.9,3);
\draw[color=blue,very thick,solid] (4.9,3) parabola bend (3.97,2.2) (3.97,2.2);

\draw[color=red,very thick,solid] (3.93,3.35) parabola bend (4.58,3.75) (4.75,2.98);

\draw[color=blue,very thick,solid] (3.97,2.2) parabola bend (3.97,2.2) (5.1,1.9);
\draw[color=blue,very thick,solid] (5.1,1.9) parabola bend (3.97,1.8) (3.97,1.8);

\draw[color=blue,very thick,solid] (3.97,1.8) parabola bend (3.97,1.8) (4.9,1);
\draw[color=blue,very thick,solid] (4.9,1) parabola bend (4.9,1) (3.7,1.5);

\draw[color=red,very thick,solid] (4.75,2.8) parabola bend (4.75,2.8) (5.34,2);
\draw[color=red,very thick,solid] (5.34,2) parabola bend (4.75,1.2) (4.75,1.2);

\draw[color=blue,very thick,solid] (3.7,1.5) parabola bend (3.7,1.5) (4,0.55);
\draw[color=blue,very thick,solid] (4,0.55) parabola bend (4,0.55) (3.1,1.15);

\draw[color=red,very thick,solid] (4.75,1.02) parabola bend (4.6,0.45) (3.98,0.7);

\draw[color=blue,very thick,solid] (3.1,1.15) parabola bend (3,0.2) (3,0.2);
\draw[color=blue,very thick,solid] (3,0.2) parabola bend (3,0.2) (2.5,1.01);

\draw[color=red,very thick,solid] (2.85,0.25) parabola bend (1.5,-0.2) (1.5,-0.2);

\draw[color=red,very thick,dotted] (1.3,-0.2) parabola bend (1.3,-0.2) (0.7,-0.05);

\draw[color=blue,very thick,dotted] (1.5,0.6) parabola bend (2.1,0.5) (2.1,0.5);

\draw[color=green] (6.2,-0.2) arc (-30:30:4.4);


\end{tikzpicture}

\caption{Building $A(\F)$}
  \label{fig:building}
\end{figure}

 Again, let $\F$ be a frame set and put $A=A_0^{\mathscr{F}}$. Suppose that  for $1\leq i$ we have constructed the $S$-acts  $A_{i-1}^{\mathscr{F}}$. We now let
 $A_i^{\mathscr{F}}=(A_{i-1})_1^{\mathscr{F}} $,
 where at each stage, in each set of equations, we always choose distinct variables.
  This gives us a sequence
\[A_0^{\mathscr{F}}\subseteq A_1^{\mathscr{F}}\subseteq A_2^{\mathscr{F}} \subseteq \hdots.\]
 We let \[A(\mathscr{F})=\bigcup_{i\in\mathbb{N}^0}A_i^{\mathscr{F}}.\]

Given the way we have labelled our variables, and our conventions on identification, we also have that, for any frame sets $\F$ and $\mathscr{G}$ with $\F\subseteq \mathscr{G}$, and any  $i,j\in \mathbb{N}^0$ with $i\leq j$,
 \[A^{\mathscr{F}}_i\subseteq A^{\mathscr{G}}_j\]
and consequently,
 \[A(\F)\subseteq A(\mathscr{G}).\]
 To avoid technical considerations of cardinality, we restrict our attention in Theorem~\ref{thm:towers} to finite frames.
 Indeed, for ease of application, we
 have in some sense been over generous with the nature of our extensions, so that what we have constructed for the set of all frames is not the injective hull \cite{berthiaume:1967}.

 \begin{Thm}\label{thm:towers} Let $A$ be an $S$-act and let $\F$ be a  set of finite frames. Then
 $A(\F)$ is $\F$-pure.  Further,
 $A$ is $\F$-pure  if and only if $A$ is a retract of
  $A(\F)$.
 \end{Thm}
 \begin{proof} The first statement follows from the usual finiteness arguments: any
 finite consistent set of equations over $A(\F)$ must be consistent over $A^{\mathscr{F}}_m$ for some $m$ and hence have a solution in $A^{\mathscr{F}}_{m+1}\subseteq A(\F)$.

 If $A$ is a retract of
  $A(\F)$, then Lemma~\ref{lem:oneway} gives that $A$ is $\F$-pure.
  For the converse, we apply Proposition~\ref{prop:built}.
 \end{proof}

\section{A new characterisation of coherency}\label{sec:mfp}

The aim of this section is to provide a so-called homological characterisation of coherency. That is, we characterise coherency of a monoid $S$ in terms of two classes of $S$-acts (each defined using completion of diagrams) coinciding.

Before stating our result we set up some notation. Let $\F$ be the frame set of  all mfp-frames and let $A$ be an $S$-act. We say an element
$\epsilon$
of $A(\F)$ has {\em level} $L(\epsilon)=n$, where $n\in \N^0$, if
$\epsilon\in A^{\mathscr{F}}_n\setminus A^{\mathscr{F}}_{n-1}$ and   $A^{\mathscr{F}}_{-1}$ in interpreted as $\emptyset$.

We now state the main result of this section, and devote the remainder of the section to its proof.

\begin{Thm}\label{mfp:main} The following are equivalent for monoid $S$:
\begin{enumerate}\item $S$ is right coherent;
\item every mfp-pure $S$-act is almost pure;
\item every mfp-pure $S$-act is absolutely pure.
\end{enumerate}
\end{Thm}

\begin{proof} If $S$ is right coherent, then every mfp-pure act is almost pure, since the right coherency of $S$ gives us by definition that every finitely generated subact of every finitely presented monogenic $S$-act has a finite presentation. Thus (1) implies (2) and clearly, (3) implies (2). We show that (2) implies (1). The result that (2) implies (3) then follows from Theorem~\ref{coherent monoids}.

Assume that (2) holds. Let $D$ be a finitely generated subact of a finitely presented and monogenic $S$-act $C$. By definition, we have that
$C=S/\rho$ where $\rho$ is a finitely generated right congruence on $S$, so that   $\rho=\langle H\rangle$ where $H\subseteq S\times S$ is finite. We aim to show that $D$ has  a finite  presentation and then call upon  Theorem~\ref{thm:coherency} to deduce that  $S$ is right coherent.

Without loss of generality we may assume that $D\neq\emptyset$, so that
\[D=\bigcup_{b\in I}[b]S\subseteq S/\rho=C,\]
where $I\neq\emptyset$ is finite and $[u]$ denotes the $\rho$-class of $u\in S$. Let $Z=\{ z_b:b\in I\}$ be a set of symbols in bijective correspondence with $I$ and consider
$\psi:F_S(Z)\rightarrow D$ given by
\[z_b\psi=[b].\]
To show that $D$ is finitely presented, we must show that the congruence $\ker\psi$ on $F_S(Z)$ is finitely generated.

As in Section~\ref{sec:const} we build the mfp-pure extension $D^{fp}(1)$ of $D$. Since $D$ is embedded in both $C$ and $B^{fp}(1)$, and by assumption $D^{fp}(1)$ is almost pure, the inclusion map $\iota:D\rightarrow D^{fp}(1)$ extends to an $S$-morphism $\overline{\iota}:C\rightarrow D^{fp}(1)$.

\begin{Lem}\label{lem:finitebuild} Let $\gamma\in  D^{fp}(1)$ have level $n$. Then $\gamma$ lies in a subact of $ D^{fp}(1)$ built from {\em finitely many} $\F$-extensions,  starting with $D$ as the base $S$-act.

\end{Lem}
\begin{proof} We proceed by induction. If $\gamma\in D$ the result is clear. Suppose now that $\gamma\in D^{fp}(1)_n\backslash D^{fp}(1)_{n-1}$. Then
$\gamma=(xs)$ where $(xs)$ denotes the equivalence class of $xs$ in $(D^{fp}(1)_{n-1}\cup xS)/ \rho_{\Sigma}$ for some finite consistent set of equations $\Sigma=\Sigma(\FF,\phi)$ in one variable. Since $\Sigma$ is finite, it certainly includes only finitely many equations with the form $xt=t\phi$.
Since the level of each $t\phi$ is strictly less than $n$, induction gives that the elements $t\phi$ each lie in
subacts of $ D^{fp}(1)$ built from {\em finitely many} $\F$-extensions of $D$.  The union of all those subacts gives a subact $A$ such that $\gamma$ lies in the extension of $A(\Sigma)$ of $A$. The result follows by induction.
\end{proof}

\begin{Cor}\label{lem:finitebuild} The element $[1]\overline{\iota}$ lies in $\overline{D}$, where
$\overline{D}$ is a subact of $ D^{fp}(1)$ built from {\em finitely many} $\F$-extensions of $D$.
\end{Cor}

Let $\mathcal{S}$ denote the finite set of finite consistent sets of equations $\Sigma$ used in building $\overline{D}$ from $D$. We note that each $\Sigma$ has a single variable, and all the variables are distinct. As much as possible, we suppress mention of the variable. In fact, we may in many cases omit it altogether in the sense that, for a set of equations $\Sigma=\Sigma(x)$ in one variable we may identify the congruence $\rho_{\Sigma}$ on
$F_S(x)$ with a right congruence $\rho$ on $S$.
For each $\Sigma\in \mathcal{S}$ we have by definition of
$\overline{D}$ that $B(\Sigma)$ is finitely presented. Since $B(\Sigma)$ is a subact of  $C(\Sigma)=xS/\rho_{\Sigma}$, we may drop mention of $x$ and consider $B(\Sigma)$ to be a subact of $C=S/\rho_{\Sigma}$.

For $\Sigma\in \mathcal{S}$
 choose and fix a set of symbols $\{ z_t^{\Sigma}:t\in  \mathcal{F}^1\}$ and let
 \[\psi_{\Sigma}:\bigcup_{t\in \mathcal{F}^1}z^{\Sigma}_{t}\rightarrow
\bigcup_{t\in \mathcal{F}^1}(t),\]
where $(u)$ is the $\rho_{\Sigma}$-class of $u\in S$, be given by
\[z_t\psi_{\Sigma}=(t).\]
Now let
\[\ker\psi_{\Sigma}=\langle J(\Sigma)\rangle\]
where $J(\Sigma)$ is finite  by virtue of $B(\Sigma)$ being finitely presented.

\begin{lem}\label{lem:dropping down} Let $C$ be an $S$-act and let $\Sigma=\Sigma(\FF,\phi)$ be a finite consistent set of equations in one variable over $C$. For an element $(xu)$ of
\[C(\Sigma)=(C\, {\cup}\,  xS)/\kappa_{\Sigma},\]
where $(xu)$ denotes the $\kappa_{\Sigma}$-class of $xu$, we have that $(xu)=(c)$ for some $c\in C$ if and only if $xu\,\rho_{\Sigma}\, xv\ell$ for some $xv\in \FF^1$ and $\ell\in S$.
\end{lem}
\begin{proof} Let $c\in C$. We have that $(xu)=(c)$ if and only if $xu\,\kappa_{\Sigma}\, c$. Since $xu\neq c$ that would necessitate  an
$H(\Sigma)\cup K(\Sigma)$-sequence
\[xu=\alpha_1t_1,\, \beta_1t_1=\alpha_2 t_2,\cdots,
\beta_nt_n=c\]
for some $n\in\N$, $(\alpha_i,\beta_i)\in H(\Sigma)\cup K(\Sigma)$ and $t_i\in S$, for $1\leq i\leq n$. Clearly
 $(\alpha_n,\beta_n)\in K(\Sigma)$; let $k$ be the least such that $(\alpha_k,\beta_k)\in K(\Sigma)$. Then
 $(\alpha_k,\beta_k)=(xv,v\phi)$ for some $xv\in \FF^1$, and
 $xu\,\rho_{\Sigma}\, xvt_k$, completing the argument.
\end{proof}

We now suppress the mention of the variables in our sets of equations.
A {\em widget} is a pair $(\gamma,h)$ where $\gamma\in \overline{D}$ and $h\in S$; the {\em level of a widget} $L=L(\gamma,h)$ is the  level $L(\gamma)$ of its first co-ordinate. If $(\gamma,h)$ is a level $n$ widget, where $n\in\N^0$, then $\gamma h$ has level $m$ for some
$0\leq m\leq n$. We say that a widget $(\gamma,h)$ is {\em stable} if $\gamma $ has the same level as $\gamma h$. If $(\gamma,h)$ is not stable, then from Lemma~\ref{lem:dropping down} we must have
that $\gamma=(c)$, where $(c)$ is the $\rho_{\Sigma}$-class of some $\Sigma=\Sigma(\FF,\phi)\in \mathcal{S}$, and $ch\,\rho_{\Sigma}\, vk$ for some  $v\in \FF^1$ and $k\in S$. Putting $\delta=v\phi$ we note that $(\delta, k)$ is itself a widget and in $\overline{D}$ we have $\gamma h=\delta k$. We say that
the widget $(\gamma,h)$ {\em descends} to the widget $(\delta,k)$ and write $(\gamma,h)\rightarrow (\delta,k)$. A {\em widget descent} is a finite sequence of widget descents
\[(\gamma_1,h_1)\rightarrow (\gamma_2,h_2)\rightarrow\hdots
\rightarrow (\gamma_{\ell},h_{\ell})
\]
where $(\gamma_{\ell},h_{\ell})$ is stable. Notice that each widget has a widget descent. We choose and fix a widget descent for each widget. Starting from level 0 widgets, we may do this in such a way that if
\[(\gamma_1,h_1)\rightarrow (\gamma_2,h_2)\rightarrow\cdots
\rightarrow (\gamma_{\ell},h_{\ell})
\]
is the  fixed widget descent for $(\gamma_1,h_1)$, then
for any $2\leq i\leq \ell$ we have that
\[(\gamma_i,h_i)\rightarrow (\gamma_{i+1},h_{i+1})\rightarrow\cdots
\rightarrow (\gamma_{\ell},h_{\ell})
\]
is the  fixed widget descent for $(\gamma_i,h_i)$.

We now define a finite set of widgets $\mathcal{W}$ which will be used to construct a set of generators of $\ker\psi$. We do this by adding finitely many elements, in finitely many stages, to $\mathcal{W}$, starting with the empty set.

Let $\sigma=[1]\overline{\iota}$. For each $(u,v)\in H$
and $b\in I$ we put
\[(\sigma, u),(\sigma,v), (\sigma,b)\mbox{ into }\mathcal{W}.\]
For each $\Sigma=\Sigma(\FF,\phi)\in\mathcal{S}$ and each $(z^{\Sigma}_th,
z^{\Sigma}_uk)\in J(\Sigma)$ we let
\[(t\phi,h),(u\phi,k)\in\mathcal{W}. \]
For each of the widgets $(\gamma,h)$ we have added to $\mathcal{W}$, we now
add to $\mathcal{W}$ all the widgets in the fixed, chosen, descent of $(\gamma,h)$. This yields a finite set of widgets
$\mathcal{W}$. Let $\mathcal{W}_0$ be the set of level 0 widgets in $\mathcal{W}$.
For $\gamma\in D$ we let
\[\gamma=[s(\gamma)q_{\gamma}],\]
where $s(\gamma)\in I$ and $q_{\gamma}\in S$.

Let
\[\mathcal{V}_1=\{ \big(z_{s(\gamma)}q_{\gamma}h, z_{s(\delta)}q_{\delta}k\big):(\gamma,h),(\delta,k)\in \mathcal{W}_0, \gamma h=\delta k\}.\]
For any $b\in I$ we have the fixed widget descent starting from
$(\sigma ,b)$. Since
\[\sigma b=[1]\overline{\iota}b=[b]\overline{\iota}=[b]\iota
=[b]\in D,\]
the widget $(\sigma ,b)$ has a widget descent terminating in a stable widget $(\gamma(b), p(b))$. In particular,
$[b]=\sigma b=\gamma(b) p(b)$. We now let
\[\mathcal{V}_2=\{ (z_b,z_{s(\gamma(b))}q_{\gamma(b)}p(b)):b\in I\}\]
and let
\[\mathcal{V}=\mathcal{V}_1\cup\mathcal{V}_2.\]

\begin{lem}\label{lem:easyinclusion}
We have that $\mathcal{V}\subseteq \ker\psi$.
\end{lem}
\begin{proof} Let $(z_{s(\gamma)}q_{\gamma}h, z_{s(\delta)}q_{\delta}k\big)\in \mathcal{V}_1$ be such that $(\gamma,h),(\delta,k)\in \mathcal{W}_0$ and $ \gamma h=\delta k$. Then $$(z_{s(\gamma)}q_{\gamma}h)\psi=(z_{s(\gamma)}q_{\gamma})\psi h=[s(\gamma)q_{\gamma}]h=\gamma h$$ and similarly, $$(z_{s(\delta)}q_{\delta}k\big)\psi=(z_{s(\delta)}q_{\delta}\big)\psi k=[s(\delta)q_{\delta}]k=\delta k$$
As $\gamma h=\delta k$, we have $(z_{s(\gamma)}q_{\gamma}h, z_{s(\delta)}q_{\delta}k\big)\in \ker\psi$, so that $\mathcal{V}_1\subseteq \ker\psi$.

To show $\mathcal{V}_2\subseteq \ker \psi$, we let $(z_b,z_{s(\gamma(b))}q_{\gamma(b)}p(b))\in \mathcal{V}_2$ with $b\in I$. Then $$(z_{s(\gamma(b))}q_{\gamma(b)}p(b))\psi=[s(\gamma(b))q_{\gamma(b)}]p(b)=\gamma(b)p(b)=[b]=(z_b)\psi$$ implying $(z_b,z_{s(\gamma(b))}q_{\gamma(b)}p(b))\in \ker \psi$, so that $\mathcal{V}_2\subseteq \ker\psi$.  Therefore,  $\mathcal{V}\subseteq \ker\psi$, as required.
\end{proof}

Our aim now is to show the converse to Lemma~\ref{lem:easyinclusion}, namely that $\ker\psi\subseteq
\langle \mathcal{V}\rangle$. To this end we need some further terminology.

\begin{Def}\label{def:widgetsequence} Let $n\in\N^0$. A
{\em $\mathcal{W}$-widget sequence} connecting
\[(\delta_0,k_0s_0)\mbox{ to }(\gamma_{n+1}, h_{n+1} s_{n+1}) \]
is a sequence
\[ (\delta_0,k_0s_0)=(\gamma_1, h_1 s_1),\,
(\delta_1,k_1s_1)=(\gamma_2, h_2 s_2),\, \cdots,
\, (\delta_n,k_ns_n)=(\gamma_{n+1}, h_{n+1} s_{n+1})\]
where:
\[\begin{array}{ll}
(\delta_i,k_i)& \mbox{are widgets in }\mathcal{W},\, 0\leq i\leq n\\
(\gamma_j,h_j)&\mbox{are widgets in }\mathcal{W},\, 1\leq j\leq n+1\\
\delta_0k_0,\gamma_{n+1}h_{n+1}&\mbox{are elements of }D\\
\gamma_ih_i=\delta_ik_i& 1\leq i\leq n.
\end{array}\]
The {\em level $L$} of a $\mathcal{W}$-widget sequence is the level of the greatest $\delta_i$ (where $0\leq i\leq n$) and the {\em value} of a $\mathcal{W}$-widget sequence is $(L,\ell)$, where $L$ is the level, and $\ell$ is the number indices $i\in\{ 0,\cdots, n\}$ such that $\delta_i$ has level $L$.
\end{Def}

In what follows, values of $\mathcal{W}$-widget sequences are ordered lexicographically.

\begin{Lem} Let $(\delta_0,k_0s_0)$ and $(\gamma_{n+1}, h_{n+1} s_{n+1})$ be connected via a  $\mathcal{W}$-widget sequence as in Definition~\ref{def:widgetsequence}. Suppose that the level of this sequence is strictly greater than $0$. Then there is a $\mathcal{W}$-widget sequence of lower value connecting  $(\delta_0',k_0's_0)$ and $(\gamma_{n+1}', h_{n+1}'s_{n+1})$, where $(\delta_0',k_0')$ is in the fixed descent of the widget $(\delta_0,k_0)$ and  $(\gamma_{n+1}', h_{n+1}')$ is in the fixed descent of the widget $(\gamma_{n+1}, h_{n+1})$ (including  the possibility they are unchanged).
\end{Lem}

 We begin by outlining the strategy of the proof.
Let us abbreviate our $\mathcal{W}$-widget sequence as
\[w_0, w_1,\cdots, w_n\]
where
\[w_i=(\delta_i,k_is_i)=(\gamma_{i+1}, h_{i+1} s_{i+1}),\]
for $1\leq i\leq n$. We pick an $i\leq j$ such that $w_i,w_{i+1},\cdots, w_j$ have highest level, and either $w_{i-1}$ has lower level, or $i=0$, and
either $w_{j+1}$ has lower level, or $j=n$.  We then `pull down' the subsequence
$w_i, \cdots, w_j$ to a sequence of widgets $w_{i}'=v_\ell, v_{\ell+1}, \cdots, v_{m}=w_{j}'$ such that we have a new $\mathcal{W}$-widget sequence
\[w_0, w_1,\cdots, w_{i-1}, w_i'=v_\ell, v_{\ell+1}, \cdots, v_{m}=w_{j}', w_{j+1}, \cdots ,w_n \]
with lower value.  This is illustrated in Figure~\ref{fig:widgets}.

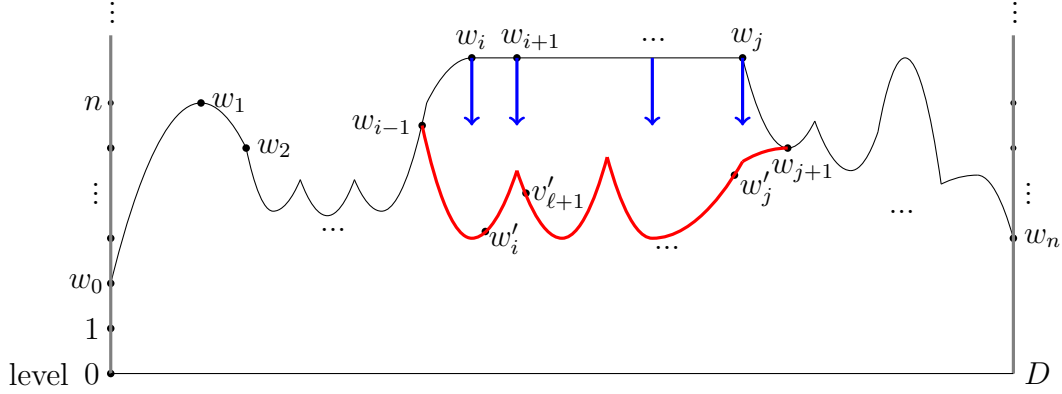
\begin{figure}[!ht]
\begin{tikzpicture}
  \draw[help lines,color=white!30,step=0.5*1.2cm] (-5*1.2,-0.5*1.2) grid (5*1.2,3*1.2);
  \fill (-5*1.2,-0.5*1.2) circle (.05);

  \fill (-5*1.2,0.5*1.2) circle (.05);
  \fill (1*1.2-5*1.2,2.5*1.2) circle (.05);
  \fill (1.5*1.2-5*1.2,2*1.2) circle (.05);
  \fill (3.45*1.2-5*1.2,2.25*1.2) circle (.05);
  \fill (4*1.2-5*1.2,3*1.2) circle (.05);
  \fill (4.5*1.2-5*1.2,3*1.2) circle (.05);
  \fill (7*1.2-5*1.2,3*1.2) circle (.05);
  \fill (7.5*1.2-5*1.2,2*1.2) circle (.05);
  \fill (10*1.2-5*1.2,1*1.2) circle (.05);

  \fill (4.15*1.2-5*1.2,1.075*1.2) circle (.05);
  \fill (4.6*1.2-5*1.2,1.5*1.2) circle (.05);
  \fill (6.91*1.2-5*1.2,1.7*1.2) circle (.05);
  \fill (0-5*1.2,0*1.2) circle (.05);
  \fill (0-5*1.2,1*1.2) circle (.05);
  \fill (-5*1.2,2*1.2) circle (.05);
  \fill (-5*1.2,2.5*1.2) circle (.04);

  \fill (10*1.2-5*1.2,2*1.2) circle (.04);
  \fill (10*1.2-5*1.2,2.5*1.2) circle (.04);

  \node at (-0.5-5*1.2,-0.5*1.2) [right] {$0$};
  \node at (-0.5-5*1.2,2.5*1.2) [right] {$n$};
  \node at (10*1.2-5*1.2,-0.5*1.2) [right] {$D$};
  \node at (-0.5-5*1.2,0*1.2) [right] {$1$};
  \node at (-1.5-5*1.2,-0.5*1.2) [right] {level};

  \node at (-0.7-5*1.2,0.5*1.2) [right] {$w_0$};
  \node at (1*1.2-5*1.2,2.5*1.2) [right] {$w_1$};
  \node at (1.5*1.2-5*1.2,2*1.2) [right] {$w_2$};
  \node at (2.55*1.2-5*1.2,2.25*1.2) [right] {$w_{i-1}$};
  \node at (3.7*1.2-5*1.2,3.2*1.2) [right] {$w_i$};
  \node at (4.2*1.2-5*1.2,3.2*1.2) [right] {$w_{i+1}$};
  \node at (5.75*1.2-5*1.2,3.2*1.2) [right] {...};
  \node at (6.75*1.2-5*1.2,3.2*1.2) [right] {$w_j$};
  \node at (7.2*1.2-5*1.2,1.75*1.2) [right] {$w_{j+1}$};
  \node at (10*1.2-5*1.2,1*1.2) [right] {$w_n$};
  \node at (8.5*1.2-5*1.2,1.3*1.2) [right] {...};
  \node at (2.2*1.2-5*1.2,1.1*1.2) [right] {...};
  \node at (5.9*1.2-5*1.2,0.9*1.2) [right] {...};
  \node at (4*1.22-5*1.2,1*1.2) [right] {$w_i'$};
  \node at (4.54*1.2-5*1.2,1.5*1.2) [right] {$v_{\ell+1}'$};
  \node at (6.85*1.2-5*1.2,1.55*1.2) [right] {$w_j'$};

  \node at (-0.4-5*1.2,1.4*1.2) [right] {.};
  \node at (-0.4-5*1.2,1.5*1.2) [right] {.};
  \node at (-0.4-5*1.2,1.6*1.2) [right] {.};
  \node at (-0.18-5*1.2,3.4*1.2) [right] {.};
  \node at (-0.18-5*1.2,3.5*1.2) [right] {.};
  \node at (-0.18-5*1.2,3.6*1.2) [right] {.};

  \node at (10*1.2-5*1.2,1.4*1.2) [right] {.};
  \node at (10*1.2-5*1.2,1.5*1.2) [right] {.};
  \node at (10*1.2-5*1.2,1.6*1.2) [right] {.};
  \node at (9.85*1.2-5*1.2,3.4*1.2) [right] {.};
  \node at (9.85*1.2-5*1.2,3.5*1.2) [right] {.};
  \node at (9.85*1.2-5*1.2,3.6*1.2) [right] {.};

  \draw(-5*1.2,0.5*1.2) parabola bend (1*1.2-5*1.2,2.5*1.2) (1.5*1.2-5*1.2,2*1.2);
  \draw(1.5*1.2-5*1.2,2*1.2) parabola bend (1.8*1.2-5*1.2,1.3*1.2) (2.1*1.2-5*1.2,1.65*1.2);
  \draw(2.1*1.2-5*1.2,1.65*1.2) parabola bend (2.4*1.2-5*1.2,1.25*1.2) (2.7*1.2-5*1.2,1.65*1.2);
  \draw(2.7*1.2-5*1.2,1.65*1.2) parabola bend (3*1.2-5*1.2,1.3*1.2) (3.5*1.2-5*1.2,2.5*1.2);
  \draw(3.5*1.2-5*1.2,2.5*1.2) parabola bend (4*1.2-5*1.2,3*1.2) (7*1.2-5*1.2,3*1.2);
  \draw(7*1.2-5*1.2,3*1.2) parabola bend (7.5*1.2-5*1.2,2*1.2) (7.8*1.2-5*1.2,2.3*1.2);
  \draw(7.8*1.2-5*1.2,2.3*1.2) parabola bend (8.2*1.2-5*1.2,1.75*1.2) (8.2*1.2-5*1.2,1.75*1.2);
  \draw(8.2*1.2-5*1.2,1.75*1.2) parabola bend (8.2*1.2-5*1.2,1.75*1.2) (8.5*1.2-5*1.2,2.18*1.2);
  \draw(8.5*1.2-5*1.2,2.18*1.2) parabola bend (8.8*1.2-5*1.2,3*1.2) (9.2*1.2-5*1.2,1.6*1.2);
  \draw(4.2*1.2,1.6*1.2) parabola bend (4.6*1.2,1.7*1.2) (4.6*1.2,1.7*1.2);
  \draw(4.6*1.2,1.7*1.2) parabola bend (4.6*1.2,1.7*1.2) (5*1.2,1*1.2);

  \draw[color=gray,very thick,solid](-5*1.2,-0.5*1.2)--(-5*1.2,3.25*1.2);
  \draw[color=gray,very thick,solid](10*1.2-5*1.2,-0.5*1.2)--(10*1.2-5*1.2,3.25*1.2);
  \draw(-5*1.2,-0.5*1.2)--(10*1.2-5*1.2,-0.5*1.2);

  \draw[color=red,very thick,solid](3.45*1.2-5*1.2,2.25*1.2) parabola bend (4*1.2-5*1.2,1*1.2) (4.5*1.2-5*1.2,1.75*1.2);
  \draw[color=red,very thick,solid](4.5*1.2-5*1.2,1.75*1.2) parabola bend (5*1.2-5*1.2,1*1.2) (5.5*1.2-5*1.2,1.9*1.2);
  \draw[color=red,very thick,solid](5.5*1.2-5*1.2,1.9*1.2) parabola bend (6*1.2-5*1.2,1*1.2) (7*1.2-5*1.2,1.85*1.2);
  \draw[color=red,very thick,solid](7*1.2-5*1.2,1.85*1.2) parabola bend (7.5*1.2-5*1.2,2*1.2) (7.5*1.2-5*1.2,2*1.2);

  \draw[color=blue,very thick,solid,->](4*1.2-5*1.2,3*1.2)--(4*1.2-5*1.2,2.25*1.2);
  \draw[color=blue,very thick,solid,->](4.5*1.2-5*1.2,3*1.2)--(4.5*1.2-5*1.2,2.25*1.2);
  \draw[color=blue,very thick,solid,->](6*1.2-5*1.2,3*1.2)--(6*1.2-5*1.2,2.25*1.2);
  \draw[color=blue,very thick,solid,->](7*1.2-5*1.2,3*1.2)--(7*1.2-5*1.2,2.25*1.2);


\end{tikzpicture}

\caption{Reducing the value of a widget sequence}
  \label{fig:widgets}
\end{figure}
\begin{proof}  Let $L$ be the greatest level of $\delta_l$  occurring in the $\mathcal{W}$-widget sequence: by assumption, $L>0$. Let $i$, where
$0\leq i\leq n$, be the smallest such that the level of $\delta_i$ is $L$. We will construct a new $\mathcal{W}$-widget sequence where, in particular, $(\delta_i,k_i)$ is replaced by a new widget in its fixed descent, and where we involve no new elements of $\overline{D}$ of level higher than $L-1$.

Consider $\gamma_{i+1}h_{i+1}$. Since $\delta_i=\gamma_{i+1}$, we have $L(\gamma_{i+1})=L$ so that
$L(\gamma_{i+1}h_{i+1})=L(\delta_{i+1}k_{i+1})\leq L$.  If
$L(\gamma_{i+1}h_{i+1})=L$, then we are forced to have
$L(\delta_{i+1})=L(\gamma_{i+2})=L$. Continuing in this manner, since $L(\gamma_{n+1} h_{n+1})=0$, we arrive at $j$ where ${i+1}\leq j\leq n+1$ such that
\[ L=L(\gamma_l)=L(\gamma_lh_l)=L(\delta_l)=L(\delta_lk_l) \]
for $i+1\leq l<j$ but \[L(\gamma_j h_j)<L=L(\gamma_j).\]
We remark that in the degenerate case where $n=0$ and so
$(\delta_0,k_0s_0)=(\gamma_{1}, h_{1} s_{1})$, then as
$L(\delta_0)=L(\gamma_1)=L>0$, and $L(\delta_0k_0)=L(\gamma_1h_1)=0$, in this case, $i=0$ and $j=1$.

From above, we have that $\gamma_{i+1}$, together with
\[\gamma_{i+1}, \delta_{i+1}=\gamma_{i+2}, \cdots, \delta_{j-1}=\gamma_{j} \]
and if $i+1<j$
\[\gamma_{i+1}h_{i+1}=\delta_{i+1}k_{i+1}, \cdots,
\gamma_{j-1}h_{j-1}=\delta_{j-1}k_{j-1},\]
all have level $L$. Given the equalities, and the construction of
$D^{fp}(1)$, this can only happen if
\[\gamma_{a}=(c_a), \, i+1\leq a\leq j\]
and
\[ \delta_b=(d_b),\, i+1\leq b\leq j-1,\]
where $(u)$ denotes the $\rho_{\Sigma}$-class of $u\in S$ for some
$\Sigma=\Sigma(\FF,\phi)\in\mathcal{S}$.
It follows that
\[ c_{a}h_{a}\,\rho_{\Sigma}\, d_{a}k_{a},\, i+1\leq a\leq j-1\]
and then, using the definition of a $\mathcal{W}$-widget sequence,
\[c_{i+1}h_{i+1}s_{i+1}\,\rho_{\Sigma}\, d_{i+1}k_{i+1}s_{i+1}
\,\rho_{\Sigma}\, c_{i+2}h_{i+2}s_{i+2}\,\rho_{\Sigma}\, d_{j-1}k_{j-1}s_{j-1}.\]
If $i>0$ then we notice that $L(\delta_{i-1})=L(\gamma_i)<L$ and so $L(\delta_{i}k_{i})=L(\gamma_ih_i)<L$. Clearly $0=L(\delta_{i}k_{i})<L$  is immediately true if $i=0$, by our assumptions on the end points of the
$\mathcal{W}$-sequence. Now
from the fact $L(\delta_ik_i)<L(\delta_i)=L$ and
$L(\gamma_j h_j)<L=L(\gamma_j)$, we have widget descents, as the first steps in our fixed, chosen, widget descents
\[(\delta_i,k_i)\rightarrow (\delta_i',k_i')\mbox{ and }
(\gamma_j,h_j)\rightarrow (\gamma_j',h_j').\]
The construction of $\mathcal{W}$ tells us that
$(\delta_i',k_i'), (\gamma_j',h_j')\in \mathcal{W}$.
By choice of our descent sequences, $(\delta_i',k_i')$
and $(\gamma_j',h_j')$ are obtained from
\[d_ik_i\,\rho_{\Sigma}\,  vk_i'\mbox{ and } c_jh_j\,\rho_{\Sigma}\, wh_j'\]
where $v\phi=\delta_i'$ and $w\phi=\gamma_j'$ for some
$v,w\in \FF^1$. This now gives us, together with earlier statements, that
\[vk_i's_i\,\rho_{\Sigma}\, d_ik_is_i\,\rho_{\Sigma}\, c_{i+1}h_{i+1}s_{i+1}\,\rho_{\Sigma}\,d_{j-1}k_{j-1}s_{j-1}\,\rho_{\Sigma}\, c_jh_js_j
\,\rho_{\Sigma} \, wh_j's_j.\]

A consequence of this is that
\[(z^{\Sigma}_vk_i's_i)\psi_{\Sigma}=(z^{\Sigma}_wh_j's_j)\psi_{\Sigma}\]
and so there is a $J(\Sigma)$-sequence
\[z^{\Sigma}_vk_i's_i=U_1t_1, V_,t_1=U_2t_2,\cdots\, , V_mt_m=z^{\Sigma}_wh_j's_j,\]
where $m\in \N^0$, $(U_i,V_i)=(z^{\Sigma}_{u(i)}{u_i},z^{\Sigma}_{v(i)}v_i)\in J(\Sigma)$ and  $t_i\in S$ for
$1\leq i\leq m$. Notice that from our choice of $\mathcal{W}$, we have that $(u(i)\phi,u_i),(v(i)\phi, v_i)\in \mathcal{W}$
for $1\leq i\leq m$, and these widgets all have level strictly less than $L$. If $m=0$ we immediately have that $v=w$ and
$k_i's_i=h_j's_j$. If $m\geq 0$ we   have the following sequences of equalities:
\[ v=u(1),v(1)=u(2),\,\cdots\, ,v(m)=w\]
and
\[k_i's_i=u_1t_1, \, v_1t_1=u_2t_2,\cdots\, ,v_mt_m=h_j's_j.\]
Finally, since $(U_i,V_i)\in J(\Sigma)$ we have
\[u(i)u_i\,\rho_{\Sigma}\, v(i)v_i\,\] so that
Proposition \ref{prop:easy} gives us that
\[u(i)\phi \, u_i=v(i)\phi \, v_i\]
for $1\leq i\leq m$.

We observe that if $m=0$, then $\delta_i'=v\phi=w\phi=\gamma_j'$,  and otherwise, $\delta_i'=v\phi=u(1)\phi$ and
$v(m)\phi=w\phi=\gamma_j'$. We can now write down our new $\mathcal{W}$-widget sequence:

\[(\delta_0,k_0s_0)=(\gamma_1, h_1 s_1),\,
(\delta_1,k_1s_1)=(\gamma_2, h_2 s_2),\, \cdots,
\,
(\delta_{i-1},k_{i-1}s_{i-1})=(\gamma_i, h_{i} s_{i}),\]\[
(\delta_i', k_i's_i)=(u(1)\phi , \, u_1t_1),
(v(1)\phi, v_1t_1)=(u(2)\phi,u_2t_2),\cdots\, , (v(m)\phi,v_mt_m)=(\gamma_j',h_j's_j),  \]
\[(\delta_j,k_js_j)=(\gamma_{j+1},h_{j+1}s_{j+1}), \cdots,
(\delta_n,k_ns_n)=(\gamma_{n+1}, h_{n+1} s_{n+1}).\]

If $i=0$ or $j=n+1$, then we have changed the end-points in the prescribed way. Notice that our new $\mathcal{W}$-widget sequence has value strictly less than the original.

\end{proof}

Induction now yields the following.

\begin{Cor}\label{cor:thefloor} Let $(\delta_0,k_0s_0)$ and $(\gamma_{n+1}, h_{n+1} s_{n+1})$ be connected via a  $\mathcal{W}$-widget sequence as in Definition~\ref{def:widgetsequence}. Suppose that the level of this sequence is strictly greater than $0$. Then there is a $\mathcal{W}$-widget sequence of level $0$  connecting  $(\delta_0'',k_0''s_0)$ and $(\gamma_{n+1}'', h_{n+1}''s_{n+1})$, where $(\delta_0'',k_0'')$ is  the final term of the fixed descent of the widget $(\delta_0,k_0)$ and  $(\gamma_{n+1}'', h_{n+1}'')$ is the final term of the fixed descent of the widget $(\gamma_{n+1}, h_{n+1})$.
\end{Cor}
\begin{proof} We begin by removing the widgets of highest level, until they are all removed. To lower the value of the sequence further, we must remove the widgets of the next highest value. We continue until the level of the $\mathcal{W}$-widget sequence is $0$. At that stage the endpoints have the required form.
\end{proof}

\begin{Lem}
We have that $\ker\psi=\langle \mathcal{V}\rangle.$
\end{Lem}

\begin{proof}
By Lemma \ref{lem:easyinclusion}, it only remains to show that $\ker\psi\subseteq \langle \mathcal{V}\rangle$. Let $(z_b r)\psi=(z_c t)\psi$ where $b, c\in I$. Then $[br]=[ct]$, so that $br~\rho~ct$ and there exists $H$-sequence  $$br=u_1s_1, v_1s_1=u_2s_2, \cdots, v_ns_n=ct$$ where
$n\in \N^0$, $(u_i, v_i)\in H$ and $t_i\in S$ for all $1\leq i\leq n$. Recall that $\sigma=[1]\overline{\iota}$
and $(\sigma,b),(\sigma,c),(\sigma,u_i)$ and $(\sigma,v_i)\in \mathcal{W}$ for $1\leq i\leq n$. Observe that for
$1\leq i\leq n$ we have that
\[\sigma u_i=[1]\overline{\iota}u_i=[u_i]\overline{\iota}=
[v_i]\overline{\iota}=[1]\overline{\iota}v_i=
\sigma v_i.\] We therefore have the following $\mathcal{W}$-widget sequence $$(\sigma, br)=(\sigma, u_1s_1), (\sigma, v_1s_1)=(\sigma, u_2s_2), \cdots, (\sigma, v_ns_n)=(\sigma, ct).$$ By induction, there exists a $\mathcal{W}$-sequence $$(\gamma(b), p(b)r)=(\gamma_1, h_1s_1), (\delta_1, k_1s_1)=(\gamma_2, h_2s_2), \cdots, (\delta_n, k_ns_n)=(\gamma(c), p(c) t).$$ Then

\[z_b r~\langle\mathcal{V}\rangle~ z_{s(\gamma(b))}q_{\gamma(b)}p(b) r=z_{s(\gamma_1)}q_{\gamma_1}h_1s_1 ~\langle\mathcal{V}\rangle~z_{s(\delta_1)}q_{\delta_1}k_1s_1=z_{s(\gamma_2)}q_{\gamma_2}h_2s_2 ~\langle\mathcal{V}\rangle~\]
\[
\dots~\langle\mathcal{V}\rangle~z_{s(\delta_n)}q_{\delta_n}k_ns_n
=~z_{s(\gamma(c))}q_{\gamma(c)}p(c)t~\langle\mathcal{V}\rangle~z_c t.\]
\end{proof}

We now have completed proof of Theorem \ref{mfp:main}.
\end{proof}

\begin{Conj}\label{con:fp} We conjecture that a further equivalent conditions could be added to Theorem~\ref{mfp:main}, namely that the classes of fp-pure $S$-acts and absolutely pure $S$-acts coincide. \end{Conj}

\section{Monoids $S$ that are not right coherent such that $\mathcal{A}_S(1)=\mathcal{A}_S(\aleph_0)$}\label{sec:ex}

In light of Theorem~\ref{coherent monoids}, in particular the construction of its proof, one might wonder whether right coherency is necessary for $\mathcal{A}_S(1)=\mathcal{A}_S(\aleph_0)$. However, this is not the case.

\subsection{Monoids with the fem-property}\label{subs:fem}

\begin{Def}\label{def:femyclic} A monoid $S$ satisfies the {\em fem-property} if every finitely generated $S$-act embeds into a monogenic act.
\end{Def}

We begin with an easy observation. Note that the strategy is, in some sense, reminiscent of that of \cite{levin:64, levin:70}.

\begin{Prop}\label{prop:cyclic} Let $S$ satisfy the fem-property. Then $\mathcal{A}_S(1)=\mathcal{A}_S(\aleph_0)$.
\end{Prop}
\begin{proof} Let $\Sigma=\Sigma(X)$ be a finite consistent set of equations over an almost pure $S$-act $A$. Let $A'$ be the subact of $A$ generated by the constants appearing in the equations of $\Sigma$.  Certainly $\Sigma$ has a solution $(b_x)_{x\in X}$ in some $S$-act $B$ containing $A$ and hence $A'$. Let $B'$ be the subact of $B$ generated by $A'$ and $\{ b_x:x\in X\}$. By assumption $B'$ is a subact of a monogenic $S$-act $C=cS$. Let $b_x=cs_x$ for each $x\in X$ and let $\Pi=\Pi(w)$ be the set of equations in a single variable $w$ obtained by replacing each $xs=yt\in \Sigma$ by $ws_xs=ws_yt$ and each $zu=a\in \Sigma$ by $ws_zu=a$.  Then $c\in C$ is a solution to $\Pi$. We may amalgamate $A$ and $C$ over $A'$; call the amalgmation $D$. So, we can regard $\Pi$ as a  set of equations in one variable over $A$ with a solution in $D$. Since  $A$ is almost pure,  $\Pi$ has a solution $d\in A$. Clearly $(ds_x)_{x\in X}$ is a solution to $\Sigma$ in $A$.
\end{proof}

There are examples of right coherent monoids both with and without the fem-property. First, we characterise those monoids satisfying the fem-property.

\begin{Thm}\label{thm:fem} The following are equivalent for  a monoid $S$:
\begin{enumerate} \item $S$ satisfies the fem-property;
\item every 2-generated $S$-act embeds into a monogenic $S$-act;
\item $F_S(X)$ embeds into a monogenic $S$-act, where $|X|=2$;
\item
 $F_S(X)$ embeds into $S$,  where $|X|=2$;
\item  there exists left cancellable elements $s,t\in S$ such that $sS\cap tS=\emptyset$.
\end{enumerate}
\end{Thm}
\begin{proof}  It is clear that (1) and (2) are equivalent,  that (2) implies (3), and that (4) and (5) are equivalent.

Suppose that (3) holds.  Let $X=\{ x,y\}$ and $F_S(X)=
xS\cup yS$ such that there is an $S$-morphism $\theta: F_S(X)\rightarrow cS$ for some monogenic $S$-act $cS$.
Consider $\psi:S\rightarrow cS$, where $1\psi=c$. If $D=(xS\cup yS)\theta$, then it is easy to see that $D\psi^{-1}=x'S\cup y'S$, where $x\theta=x'\psi$ and $y\theta=y'\psi$,  is a subact of $S$ isomorphic to $F_S(X)$. Thus (4) holds.

Finally, suppose that (5) holds and $A=aS\cup bS$ is a 2-generated $S$-act. Let  $\theta:sS\cup tS\rightarrow  A$ be such that $s\theta=a$ and $t\theta=b$. Let $\kappa=\ker\theta \cup \iota_S$ where $\iota_S$ is the identity relation on $S$. It is clear that $\kappa$ is a right congruence on $S$, and $A$ embeds into $S/\kappa$. Hence (2) holds.
\end{proof}

It is clear from Theorem~\ref{thm:fem} that if $S$ is a monoid with zero, then $S$ does not have the fem-property. Neither does any inverse monoid, as any two principal right ideals of an inverse monoid have non-empty intersection.

From \cite[Corollary 5.6]{dandan:2020}, a monoid is right coherent if and only if the monoid obtained by adjoining a zero has the same property. Thus, having a zero, or not, is not significant for coherency.

On the other hand,  there are examples of monoids that are not right coherent monoid such that every act embeds into a monogenic one. From \cite{dandan:2020} we know that if $S=F_2\times F_2$ where $F_2$ is the free monoid on 2 generators, then $S$ is not right coherent.  Further, since $F_2\times F_2$ is cancellative, any two principal right ideals are isomorphic. It is  then easy to see that $S$ has property  (4) in Theorem~\ref{thm:fem}. For, if  $F_2$ is generated by $\{ a,b\}$,  we have $(a,1)S\cap (b,1)S=\emptyset$.

\subsection{Almost pure acts over the Fountain monoid}
The main result of this section, Theorem~\ref{fountain monoid}, gives an example of a  monoid  that is not right coherent, does not have the fem-property, yet  nevertheless $\mathcal{A}_S(1)=\mathcal{A}_S(\aleph_0)$. In fact, our example is almost as far from the fem-property as possible, in that it is a chain of length 5 of principal (right) ideals.

In choosing our example, we did not have a great deal of scope. As commented, many well-behaved monoids are known to be right coherent and, for those that are not, understanding the congruences on their finitely generated free acts would be hard. As mentioned above, it is known from \cite{gould:2017a} that free inverse monoids on more than one generator are not right coherent, but a full description of their right congruences is lacking. With this in mind we choose the following specific example, taken from \cite{gould:1992} and due to Fountain: we present it in a slightly different way.

\begin{Ex}\cite{gould:1992}\label{fountain} Let $G$ be an abelian group which is not finitely generated. Let $N=\{ 1,\alpha, \alpha^2,\alpha^3,\alpha^4=0\}$ be a 5-element monogenic monoid (with $\alpha$ having index 4 and period 1). Let $P=G\times N$ and define the relation $\sim $ on $P$ to be the union of equality with
\[\{ \big((g,\alpha^k), (h,\alpha^k)\big):g,h\in G, k\in \{ 3,4\}\}. \]
Notice that $\sim$ is a congruence on $P$. We let $S=P/\sim$. For convenience, we may denote $(g,\alpha^k)$ by $g\alpha^k$ or $\alpha^kg$.  We will also use Greek letters to denote elements of $S$, for example, $\beta=\alpha^3g=\alpha^3$. The element $\alpha^4$ is a zero for $S$ and we will usually denote this by $0$.
\end{Ex}

We call the  monoid  in Example~\ref{fountain} the {\em  Fountain monoid}.
As shown in  \cite{gould:1992}, the Fountain monoid is not right coherent. However, it is easy to see that its only (right) ideals are:
\[\{ 0\} \subset  \alpha^3S=\{ 0,\alpha^3\}\subset  \alpha^2S
\subset \alpha S\subset S.\]

We define two maps $$\psi: S\backslash \{0, \alpha^3\}\rightarrow \{0, 1, 2\}\mbox{~and~} \phi:  S\backslash\{0, \alpha^3\}\rightarrow G$$
by $$\beta\psi=i\mbox {~and~} \beta\phi=g, \mbox{~where~} \beta=\alpha^ig\in S\backslash\{0, \alpha^3\}.$$
For each $\beta\in S\backslash\{0, \alpha^3\}$, we therefore have $\beta=\alpha^{\beta\psi}\beta\phi$. Effectively, $\psi$ and $\phi$ are restrictions of the projection maps to the part of $S$ consisting of singleton equivalence classes, and will behave as morphisms provided products do not fall into the ideal $\alpha^3S$.

\begin{them}\label{fountain monoid}
Let $A$ be an $S$-act over the Fountain monoid $S$. Then  $A$ is almost pure if and only if it is absolutely pure.
\end{them}

\begin{proof}
Let $A$ be an almost pure $S$-act and let $\Sigma=\Sigma(X)$ be a finite consistent set  of equations over $A$. We must show that $\Sigma$ has a solution in $A$.

We proceed by induction. If  $|X|=1$, then $\Sigma$ has a solution in $A$, since $A$ is almost pure. We suppose that $|X|=n\geq 2$ and every consistent set of equations over $A$ in at most $n-1$ variables has a solution in $A$.

The first part of our strategy is to reduce the question of solubility in $A$ of $\Sigma$ to that of some `simpler' sets of equations obtained from $\Sigma$. Suppose that $\Sigma(X)$ is not connected, that is, we can write $\Sigma(X)$ as $\Sigma(Y)\cup\Sigma(Z)$ where $Y,Z$ are non-empty subsets of $X$ such that
$X=Y\dot{\cup} Z$ (so that also $\Sigma(Y)\cap\Sigma(Z)=\emptyset$). In this case, we could immediately call upon our inductive assumption to obtain a solution in $A$ to $\Sigma(Y)$ and $\Sigma(Z)$ and hence to $\Sigma(X)$. Thus, at any stage, we may assume our sets of equations are connected.

 Let $(\bar{y})_{y\in X}$ be a solution to $\Sigma(X)$ in some $S$-act $B$ containing $A.$ For each $x\in X$, let
 \[K(\bar{x})=\{ \gamma\in S: \bar{x}\gamma\in A\}.\]
 Notice that $K(\bar{x})$ is either empty or an ideal of $S$. Let $L=\{x\in X: K(\bar{x})\neq \emptyset\}$. Since each ideal of $S$ is principal, for each $x\in L$ we may fix some $\tau(x)\in S$ and $a(x)\in A$ such that $K(\bar{x})=\tau(x)S$ and  $\bar{x}\tau(x)=a(x)$. It follows that $\bar{x}S\cap A=a(x)S$. In the rest of the proof, we will always take  $\tau(x)$ to be a power of $\alpha$.

 Let $\Sigma_c(X)$ be the set of all equations of $\Sigma(X)$ involving constants, and $\Sigma_{nc}=\Sigma(X)\backslash \Sigma_c(X)$. We put
 \[\Sigma'(X)=\Sigma_{nc}(X)\cup \{x\tau(x)=a(x): x\in L\}.\] Certainly $\Sigma'(X)$ is finite and consistent with  a solution $(\bar{y})_{y\in X}$ in $B$. We claim that any solution to $\Sigma'(X)$ will be a solution to $\Sigma(X)$. Suppose that $(y^*)_{y\in X}$ is a solution to $\Sigma'(X)$. Notice first that for each $x\mu=b\in \Sigma(X)$, we have $\bar{x}\mu=b$ so that $K(\bar{x})\neq \emptyset$ and then $x\in L$. Thus, $\tau(x)$ and $a(x)$ exist with $\bar{x}\tau(x)=a(x)$  and $\mu=\tau(x)\nu$ for some $\nu\in S$.
As $(y^*)_{y\in X}$ is a solution of  $\Sigma'(X)$  we have $x^*\tau(x)=a(x)$ and so
\[b=\bar{x}\mu=\bar{x}\tau(x)\nu=a(x)\nu=x^*\tau(x)\nu=x^*\mu,\] as required. We therefore focus on finding a solution for $\Sigma'(X)$; relabelling $\Sigma'(X)$ by $\Sigma(X)$.

We proceed to eliminate some forms of $\Sigma(X)$ that are easy to handle.

Suppose that $\Sigma(X)$ contains no equations with constants. Since $A$ is almost pure, it has  local left zeros, and so from a comment following Definition~\ref{def:local}, $\Sigma(X)$ has a solution in $A$.
We suppose therefore that $\Sigma(X)$ contains at least one equation with a constant.

 Suppose that  $\Sigma(X)$ contains an equation of the form  $xg=a$ for some $g\in G$. Then  $\bar{x}=ag^{-1}\in A$
and $K(\bar{x})=S$. Replacing every $x$ in $\Sigma(X)$ by $\bar{x}$ gives a finite consistent set of equations over $A$ in $n-1$ variables with a solution $({\bar{y}})_{y\in Y},$ where $Y=X\backslash \{x\}$, so, by our inductive hypothesis, it has a solution $(\bar{\bar{y}})_{y\in Y} $ in
$A.$ Putting $\bar{\bar{x}}=\bar{x}=ag^{-1}$, we have that $(\bar{\bar{y}})_{y\in X}$ is a solution to $\Sigma(X)$ in $A$.

On the other hand, suppose that there exists $yg=z\gamma\in \Sigma(X)$ for some $y,z\in X$ with
$y\neq z$, $g\in G$ and $\gamma\in S$. Then $\bar{y}=\bar{z}hg^{-1};$ replacing every $y$ in $\Sigma(X)$ by $zhg^{-1}$ yields a consistent set of equations over $A$ in $n-1$ variables with a solution $(\bar{y})_{y\in Z}$ where $Z=X\setminus\{ y\}$. Again, by induction, it has a solution $(\bar{\bar{y}})_{y\in Z}$ in $A.$ Putting $\bar{\bar{y}}=\bar{\bar{z}}hg^{-1}$, we obtain a solution $(\bar{\bar{y}})_{y\in X}$ to $\Sigma(X)$ in $A$.

We assume therefore that  $\Sigma(X)$ contains at least one equation with a constant and there are no equations in $\Sigma(X)$ with form $xg=a$ or $xg=y\gamma$ for any $g\in G, \gamma \in S, a\in A$ and $x\neq y\in X$. Clearly, with such assumption,  $K(\bar{x})\neq S$ for each $x\in X$. For use in the later parts of the proof,
we define three disjoint copies of $X$ as follows:
 $$X^0=\{x^0: x\in X\}, X^1=\{x^1: x\in X\}, X^2=\{x^2: x\in X\}$$
 and put $Z=X^0\cup X^1\cup X^2$.

 Assume that for any $x\in X$ with $x\tau(x)=a(x)\in \Sigma(X)$, we have $a(x)0=a(x)$. Let $x\tau(x)=a(x), y\tau(y)=a(y)\in \Sigma(X)$ for some $x, y\in X$. Since $\Sigma(X)$ is connected, it follows that $x0 \,\rho_{\Sigma}\, y0$  so that $\bar{x}\mu=\bar{y}\nu$. Then
$$a(x)=a(x)0=\bar{x}\tau(x) 0=\bar{x}0=\bar{y}0=\bar{y}\tau(y) 0=a(y)0=a(y).$$
Hence all constants appearing in $\Sigma(X)$ are equal. Let $a(x)$ be one such constant. Since $a(x)t=a(x)$ for all $t\in S$, we deduce that $(\bar{\bar{y}})_{y\in X}$ is a solution to $\Sigma(X)$ where $\bar{\bar{y}}=a(x)$ for all $y\in X.$

 Suppose from now on that there exists some $y\tau(y)=a(y)\in \Sigma(X)$ where $a(y)0\neq a(y).$ Let $$W=\{x\in X: K(\bar{x})\neq 0, a(x)0\neq a(x)\}.$$ Pick $x\in W$ such that $K(\bar{x})=\tau(x)S$ is the largest ideal within all ideals $K(\bar{y})$ where $y\in W$. Notice that $\tau(x)\neq 0$ for all $x\in K$, for, if it did, then
 $x0=a(x)$ would give $a(x)=a(x)$, a contradiction.   Further,  $\tau(x)\neq 0$
 and $\tau(x)\neq 1$.  We therefore consider the following three cases  determined by the choice of $\tau(x)$, which themselves will require delicate argument. To simplify notation we  let $\rho=\rho_{\Sigma(X)}$.

\medskip

{\em Case  $\tau(x)=\alpha^3$.} We therefore have $x\alpha^3=a(x)\in \Sigma(X)$ with $a(x)0\neq a(x)$. We first point out some forbidden patterns.

We cannot have \begin{equation}x\alpha^2 g~\rho~y\alpha h, x\alpha^3~\rho~y\alpha^2 h\mbox{ or } x\alpha^3~\rho~y\alpha h \label{}\end{equation}  for any $y\in X, g, h\in G$. For, if we did, then we would have
\[\bar{x}\alpha^2 g=\bar{y}\alpha h, \, \bar{x}\alpha^3=\bar{y}\alpha^2 h\mbox{ or } \bar{x}\alpha^3=\bar{y}\alpha h.\]
But this would give in the first two cases that
$a(x)=\bar{x}\alpha^3=\bar{y}\alpha^2$ and in the third that
$a(x)=\bar{x}\alpha^3=\bar{y}\alpha$. If $a(x)=\bar{x}\alpha^3=\bar{y}\alpha^2$, then $\tau(y)=\alpha$ or
$\alpha^2$, and if $a(x)=\bar{x}\alpha^3=\bar{y}\alpha$ then
$\tau(y)=\alpha$. Since $K(\bar{x})\subset K(\bar{y})$ we must have  $a(y)0=a(y)$. But, in addition, either $a(x)=a(y)$ or
$a(x)=a(y)\alpha$, so that we obtain $a(x)0=a(x)$, a contradiction.

On the other hand, we cannot have
\begin{equation}x\alpha^i g~\rho~y\alpha^{i+j}h\label{}\end{equation} for $y\in X$, where $0\leq i\leq 3,\, 1\leq j\leq 4-i,\, g,h\in G.$ Otherwise,
\[x\alpha^3=x\alpha^i\alpha^{3-i}g~\rho~y\alpha^{i+j}h\alpha^{3-i}=y0\] and so $a(x)=\overline{x}\alpha^3=\overline{y}0$, implying $a(x)0=a(x)$, a contraction.

From the above forbidden patterns, we deduce that the equations in $\Sigma(X)$  involving $x$ must have one of the following forms:
$$x \alpha g=y\alpha h, x\alpha^2g=y\alpha^2 h, x\alpha^3=y\alpha^3, x0=y\gamma$$$$xg=xh, x\alpha g=x\alpha h, x\alpha^2 g=x\alpha^2h$$
where $g,h\in G, \gamma\in S$ and $x\neq y\in X$.

\medskip
Let $Y=X\backslash\{x\},$ and let $\Sigma(x)$ be the set of all equations in $\Sigma(X)$ just involving $x$ and $\Sigma(Y)$ the set of all equations in $\Sigma(X)$ just involving variables in $Y.$

Suppose that the equations involving $x$ and $y\neq x\in X$ only have forms $x\alpha^3=y\alpha^3, x0=y\gamma$, where $\gamma\in S.$
As $\Sigma(\overline{x})$ holds,  it has a solution $\bar{\bar{x}}$ in $A$ by assumption. Consider $$\Pi(Y)=\Sigma(Y)\cup \{y\alpha^3=\bar{\bar{x}}\alpha^3, y\gamma=\bar{\bar{x}}0: y\alpha^3=x\alpha^3, y\gamma=x0\in \Sigma(X)\}.$$ For all $y\alpha^3=x\alpha^3, y\gamma=x0\in \Sigma(X)$,
\[\bar{y}\alpha^3=\bar{x}\alpha^3=a(x)=\bar{\bar{x}}\alpha^3\]
and
\[\bar{y}\gamma=\bar{x}0=\bar{x}\alpha^30=a(x)0=\bar{\bar{x}}\alpha^30=\bar{\bar{x}}0.\] Thus, $(\bar{y})_{y\in Y}$ is a solution to $\Pi(Y)$, and so it has a solution $(\bar{\bar{y}})_{y\in Y}$ in $A$. It is easy to see that $(\bar{\bar{y}})_{y\in X}$ is a solution of $\Sigma(X)$ in $A$.

\medskip

Suppose therefore that there exist equations in $\Sigma(X)$ having form  $x \alpha g=y\alpha h$ or $x\alpha^2g=y\alpha^2 h$ for some $g, h\in G$ and $y\neq x\in X$.

 Let $F_S(Z)$ be the free $S$-act over $Z$, where $Z$ consists of three disjoint copies of $X$, as defined earlier. We proceed by defining three binary relations $H_1$, $H_2$ and $H_3$ on $F_S(Z)$ as follows:
 \[ \begin{array}{lll}
         H_1=\{(y^0u, y^0v), (y^1u, y^1v), (y^2u, y^2v): yu=yv\in \Sigma(X), y\in X, u, v\in G\},\\
         H_2=\{(y^1u, z^1v), (y^2u, z^2v): y\alpha u=z\alpha v\in \Sigma(X), y, z\in X, u, v\in G\},\\
         H_3=\{(y^2u, z^2v): y\alpha^2 u=z\alpha^2 v\in \Sigma(X), y, z\in X, u, v\in G\}.\end{array} . \]
Let $H=H_1\cup H_2\cup H_3$ and $\bar{\sigma}=\langle H\rangle$. Since $G$ is coherent by \cite{gould:1992},
$$\bold{r}([x^1])=\{(u, v)\in G\times G: x^1u~\bar{\sigma}~x^1v\}=\langle W_1\rangle$$ and
$$\bold{r}([x^2])=\{(u, v)\in G\times G: x^2u~\bar{\sigma}~x^2v\}=\langle W_2\rangle$$
where $W_1,W_2$ are finite.

We now claim that $\bar{x}\alpha u=\bar{x}\alpha v$ for any $(u, v)\in W_1$. For this purpose, we define $$\theta: Z\longrightarrow B$$ by $$y^0\theta=\bar{y}, y^1\theta=\bar{y}\alpha, y^2\theta=\bar{y}\alpha^2, \mbox{~for~all~} y\in  X.$$ It is easy to check that $H\subseteq \ker \theta$, and so there exists $$\bar{\theta}: F_S(Z)/\bar{\sigma}\longrightarrow B$$ defined by
$$[y^0]\bar{\theta}=\bar{y}, [y^1]\bar{\theta}=\bar{y}\alpha, [y^2]\bar{\theta}=\bar{y}\alpha^2, \mbox{~for~all~} y\in X.$$
Let $(u, v)\in W_1$. Then $x^1u~\bar{\sigma}~x^1v$, so that $$\bar{x}\alpha u=(x^1u)\bar{\theta}=(x^1v)\bar{\theta}=\bar{x}\alpha v.$$ Similarly, we can show $\bar{x}\alpha^2 p=\bar{x}\alpha^2 q$ for any $(p, q)\in W_2$.

\medskip

To find a solution to $\Sigma(X)$ in $A$, we now construct two finite set of equations $\Pi(x)$ and $\Pi(Y)$ as follows. Let
$$\Pi(x)=\Sigma(x)\cup \{x\alpha u=x\alpha v: (u, v)\in W_1\} \cup \{(x\alpha^2u=x\alpha^2 v: (u, v)\in W_2)\}.$$ Then $\Pi(\bar{x})$ holds, so that $\Pi(x)$ has a solution $\bar{\bar{x}}$ in $A$. Let $$\Pi(Y)=\Sigma(Y)\cup \{y\gamma=\bar{\bar{x}}\delta: y\gamma=x\delta\in \Sigma(X),  y \neq x\}.$$ Let $\rho'=\rho_{\Sigma(Y)}$, and so $\rho'=\rho_{\Sigma(Y)}\subseteq \rho_{\Sigma(X)}=\rho$.

We now show that $\Pi(Y)$ is consistent by considering the following cases.

\medskip

{\em Subcase (i)} $y\mu=\bar{\bar{x}}\kappa, z\nu=\bar{\bar{x}}\eta\in \Pi(Y)$ with $y\mu=x\kappa, z\nu=x\eta\in \Sigma(X)$. Suppose that $y\mu\delta~\rho'~z\nu\varepsilon$ for some $\delta, \varepsilon\in S$. Then $x\kappa\delta~\rho~y\mu\delta~\rho~z\nu\varepsilon~\rho~x\eta\varepsilon$. We consider the following subcases.

{\em Subcase (i)(a)} $\kappa\delta\in G$. This implies $\kappa\in G$, a contradiction.

 {\em Subcase (i)(b)} $\kappa\delta\in \alpha G$. By the forbidden patterns (4.1) and (4.2), we deduce $\eta\varepsilon\in \alpha G$. Let $\kappa\delta=\alpha g$ and $\eta\varepsilon=\alpha h$ for some $g, h\in G$. Then $x\alpha g~\rho~x\alpha h$, so that there exists $n\in \mathbb{N}$ and an $H(\Sigma)$-sequence $$x\alpha g=y_1u_1t_1, z_1v_1t_1=y_2u_2t_2, \cdots, z_nv_nt_n=x\alpha h$$ where $t_1, \cdots, t_n\in S$
and  $$y_1u_1=z_1v_1, \cdots, y_nu_n=z_nv_n\in \Sigma(X).$$ Again, by the forbidden patterns 4.1 and 4.2, we have $u_it_i, v_it_i\in \alpha G$ for all $1\leq i\leq n$, so that $u_i, v_i\in G$ or $u_i, v_i\in \alpha G$. Notice that $u_i, v_i\in G$ happens only if $y_i=z_i$ by assumption. Therefore, we have $(y_i^1(u_i)\phi, z_i^1(v_i)\phi)\in H_1\cup H_2$, implying $$(y_i^1(u_i)\phi(t_i)\phi, z_i^1(v_i)\phi(t_i)\phi)\in  \bar{\sigma}$$ and so $(y_i^1(u_it_i)\phi, z_i^1(v_it_i)\phi)\in  \bar{\sigma}$.

On the other hand, since the identities involving in the above $H(\Sigma)$-sequence holds in $F_S(Z)$, we have $$x^1 g=y_1^1(u_1t_1)\phi, z_1^1(v_1t_1)\phi=y_2^1(u_2t_2)\phi, \cdots, z_n^1(v_nt_n)\phi=x^1h$$ in $F_S(Z).$ Therefore, $x^1g~\bar{\sigma}~x^1h$ and so $(g, h)\in \bold{r}([x^1])$.  Then there exists $n\in \mathbb{N}$ and a $W_1$-sequence such that
\[g=p_1s_1, q_1s_1=p_2s_2, \cdots, q_ns_n=h\]
where  $s_1, \cdots, s_n\in G$ and $(p_i, q_i)\in W_1$ for all $1\leq i\leq n$. By the construction of $\Pi(x)$, we have $\bar{\bar{x}}\alpha p_i=\bar{\bar{x}}\alpha q_i$ for all $1\leq i\leq n$,  so that
\[\bar{\bar{x}}\alpha g=\bar{\bar{x}}\alpha p_1s_1=\bar{\bar{x}}\alpha q_1s_1=\cdots=\bar{\bar{x}}\alpha q_ns_n=\bar{\bar{x}}\alpha h\] and so
\[\bar{\bar{x}}\kappa\delta=\bar{\bar{x}}\alpha g=\bar{\bar{x}}\alpha h=\bar{\bar{x}}\eta\varepsilon\] as required.

 {\em Subcase (i)(c)}  $\kappa\delta\in \alpha^2 G$. In this case, we must have $\eta\varepsilon\in \alpha^2 G$. Let $\kappa\delta=\alpha^2 g$ and $\eta\varepsilon=\alpha^2 h$ for some $g, h\in G$.  By similar argument to that of {\em Subcase (i)(b)}, this time using the construction of $\sigma_2$ and the fact that $\bar{\bar{x}}\alpha^2u=\bar{\bar{x}}\alpha^2v$ for all $(u, v)\in W_2$, we can show that $\bar{\bar{x}}\kappa\delta=\bar{\bar{x}}\alpha^2 g=\bar{\bar{x}}\alpha^2 h=\bar{\bar{x}}\eta\varepsilon$.

 {\em Subcase (i)(d)}  $\kappa\delta=\eta\varepsilon=\alpha^3$ or $\kappa\delta=\eta\varepsilon=0$. As $\kappa\delta=\eta\varepsilon$, $\bar{\bar{x}}\kappa\delta=\bar{\bar{x}}\eta\varepsilon$.

 {\em Subcase (ii)} $z\nu=c\in \Sigma(Y)$ and $y\mu=\bar{\bar{x}}\kappa\in \Pi(Y)$ with $y\mu=x\kappa\in \Sigma(X)$. Suppose that $y\mu\delta~\rho'~z\nu\varepsilon$, and so $x\kappa\delta~\rho~y\mu\delta~\rho~z\nu\varepsilon$. Then $\bar{x}\kappa\delta=\bar{z}\nu\varepsilon=c\varepsilon$, so that $\kappa\delta=\tau(x)\gamma$ for some $\gamma\in S$. Notice that $x\tau(x)=a(x)\in \Sigma(x)$, so that $\bar{x}\tau(x)=a(x)=\bar{\bar{x}}\tau(x)$, and hence $$c\varepsilon=\bar{x}\kappa\delta=\bar{x}\tau(x)\gamma=\bar{\bar{x}}\tau(x)\gamma=\bar{\bar{x}}\kappa\delta.$$

 {\em Subcase (iii)}  $y\mu=b, z\nu=c\in \Sigma(Y)$. Let $y\mu\delta~\rho'~z\nu\varepsilon$ for some $\delta, \varepsilon\in S$. Then $y\mu\delta~\rho~z\nu\varepsilon$, and so $b\delta=c\varepsilon$ by Proposition \ref{prop:soln}.

\medskip

Therefore, $\Pi(Y)$ is consistent, so that it has a solution $(\bar{\bar{y}})_{y\in Y}$ in $A$ by induction, and hence $(\bar{\bar{y}})_{y\in X}$ is a solution to $\Sigma(X)$ in $A$.

\medskip

{\em Case $\tau(x)=\alpha^2$. } We therefore have $x\alpha^2=a(x)\in \Sigma(X)$ with $a(x)0\neq a(x).$  We first point out some forbidden patterns.

We cannot have
\begin{equation}x\alpha^ig\,\rho\, x\alpha^{i+j}h\label{}\end{equation}
for any $0\leq i\leq 2$ and any $1\leq j$. For, if we did, then multiplying by a suitable power of $\alpha$ would give
\[x\alpha^2g\,\rho\, x\alpha^{2+j}h.\]
But then it follows that

\[x\alpha^2g\,\rho\, x\alpha^{2+kj}h(g^{-1}h)^{k-1}\]
for any $0\leq k$. It follows that  $x\alpha^2\,\rho\, x0$, giving the contradiction that $a(x)0= a(x)$. Hence any equations of $\Sigma(x)$ must have one of the following forms
$$xg=xh, x\alpha g=x\alpha h, x\alpha^2g=x\alpha^2h, x\alpha^3=x0.$$

For $y\neq x$ we cannot have
\begin{equation}x\alpha g~\rho~y\alpha^3, x\alpha g~\rho~ y 0, x\alpha^2 g~\rho~y 0\label{}\end{equation} for any $g\in G$, as it would give $x\alpha^2 g~\rho~y 0$ and then
\[a(x)=\bar{x} \alpha^2 =\bar{x}\alpha^2 gg^{-1}=\bar{y}0\] and so $a(x)0=a(x)$, a contradiction.

 We cannot have any of the following:
\begin{equation}x\alpha g~\rho~yh~\rho~y\alpha s, x\alpha g~\rho~yh~\rho~y\alpha^2 s \mbox{ or } x\alpha g~\rho~y\alpha h~\rho~y\alpha^2 s \label{}\end{equation} \mbox{ or } \begin{equation} x\alpha^2 g~\rho~yh~\rho~y\alpha s,  x\alpha^2 g~\rho~yh~\rho~y\alpha^2 s \mbox{ or } x\alpha^2 g~\rho~y\alpha h~\rho~y\alpha^2 s\label{}\end{equation} for some $g, h, s\in G$, as any  of these would give $x\alpha^2~\rho~x0$, yielding the contradiction that $a(x)0=a(x)$.
However, notice that we are not ruling out $x\alpha^2 g~\rho~y\alpha^3$ for some $g\in G$.

We now define 4 binary relations on $F_S(Z)$ as follows:
\[ \begin{array}{lll}
         P_1=\{(y^0g, y^0h), (y^1g, y^1h), (y^2g, y^2h): yg=yh\in \Sigma(X),g, h\in G\},\\
         P_2=\{(y^1g, z^1h), (y^2g, z^2h): y\alpha g=z\alpha h\in \Sigma(X),  y, z\in X, g, h\in G\},\\
         P_3=\{(y^1g, z^2h): y\alpha g=z\alpha^2 h\in \Sigma(X),  y, z\in X, g, h\in G\},\\
         P_4=\{(y^2g, z^2h): y\alpha^2 g=z\alpha^2 g\in \Sigma(X),  y, z\in X, g, h\in G\}.\end{array} \]

Let $P=P_1\cup P_2\cup P_3\cup P_4$ and $\bar{\rho}=\langle P\rangle$. Since $G$ is coherent,
$$\bold{r}([x^1])=\{(u, v)\in G\times G: x^1u~\bar{\rho}~x^1v\}=\langle Q_1\rangle$$ and
$$\bold{r}([x^2])=\{(u, v)\in G\times G: x^2u~\bar{\rho}~x^2v\}=\langle Q_2\rangle$$  where $Q_1$ and $Q_2$ are finite.

We now claim that $\bar{x}\alpha u=\bar{x}\alpha v$ for any $(u, v)\in Q_1$ and $\bar{x}\alpha^2 p=\bar{x}\alpha^2 q$ for any $(p, q)\in Q_2$. Let $\theta$ be the map defined in Case $\tau(x)=\alpha^3$.  It is easy to check that $P\subseteq \ker \theta$, and so there exists $$\bar{\theta}: F_S(Z)/\bar{\rho}\longrightarrow B$$ defined by $$[y^0]\bar{\theta}=\bar{y}, [y^1]\bar{\theta}=\bar{y}\alpha, [y^2]\bar{\theta}=\bar{y}\alpha^2, y\in X.$$
Let $(u, v)\in Q_1$. Then $x^1u~\bar{\rho}~x^1v$, so that $$\bar{x}\alpha u=(x^1u)\bar{\theta}=(x^1v)\bar{\theta}=\bar{x}\alpha v.$$ Similarly, we can show that $\bar{x}\alpha^2 p=\bar{x}\alpha^2 q$ for any $(p, q)\in Q_2$.

\medskip

Let $Y=X\backslash \{x\}$. Let
\[\Pi(x)=\Sigma(x)\cup \{x\alpha u=x\alpha v: (u, v)\in T_1\}\cup \{x\alpha^2u=x\alpha^2 v: (u, v)\in T_2\}\]$$ \ \ \ \ \ \ \ \ \ \ \  \ \ \ \ \ \ \ \ \ \ \ \ \ \ \ \ \ \ \ \ \ \  \ \ \ \ \ \ \ \ \ \ \  \ \ \ \ \ \  \  \cup \ \{x\alpha^3=x0: \mbox{ if } \bar{x}\alpha^3=\bar{x}0 \}.$$
Then $\Pi(\bar{x})$ holds, so $\Pi(x)$ has a solution $\bar{\bar{x}}$ in $A$. Let
$$\Pi(Y)=\Sigma(Y)\cup \{y\gamma=\bar{\bar{x}}\delta: y\gamma=x\delta\in \Sigma(X)\}.$$ Clearly $\rho'=\rho_{\Sigma(Y)}\subseteq \rho_{\Sigma(X)}=\rho$.
We now show that $\Pi(Y)$ is consistent.

\medskip

 {\em Subcase (i)} $y\mu=\bar{\bar{x}}\kappa, z\nu=\bar{\bar{x}}\eta\in \Pi(Y)$, where $y\mu=x\kappa, z\nu=x\eta\in \Sigma(X)$. Suppose that
$y\mu\delta\,\rho'\, z\nu\varepsilon$ for some $\delta, \varepsilon\in S$. Then  $x\kappa\delta~\rho~y\mu\delta\,\rho\, z\nu\varepsilon~\rho~x\eta\varepsilon$.

{\em Subcase (i)(a)} $\kappa\delta\in G$. In this case, we have
$\kappa\in G$ and  $x\kappa=y\mu\in \Sigma(X)$, contradicting the  forms  of equations in $\Sigma(X)$.

{\em Subcase (i)(b)}  $\kappa\delta=\alpha^3$. We know that $\eta\varepsilon=\alpha^3$ or 0. If $\eta\varepsilon=\alpha^3$ we are done. If $\eta\varepsilon=0$, then  $x\alpha^3~\rho~x0$, so that $\bar{x}\alpha^3=\bar{x}0$, and hence $\bar{\bar{x}}\alpha^3=\bar{\bar{x}}0$ by the definition of $\Pi(x)$.

{\em Subcase (i)(c)}  If $\kappa\delta=0$. Here $\eta\varepsilon$ must be $\alpha^3$ or 0. Then by a similar argument to that of  {\em Subcase (i)(b)}, we have $\bar{\bar{x}}\alpha^3=\bar{\bar{x}}0$.

{\em Subcase (i)(d)}  $\kappa\delta\in \alpha G$. In this case, we have $\eta\varepsilon\in \alpha G$. Let $\kappa\delta=\alpha g$ and $\eta\varepsilon=\alpha h$ for some $g, h\in G$. Then $x\kappa\delta~\rho~x\eta\varepsilon$, so that either
 $x\kappa\delta=x\eta\varepsilon$, or there exists $n\in \mathbb{N}$ and  an $H(\Sigma)$-sequence
 \[x\alpha g=y_1u_1t_1, z_1v_1t_1=y_2u_2t_2, \cdots, z_nv_nt_n=x\alpha h\]
 where $t_1, \cdots, t_n\in S$ and \[y_1u_1=z_1v_1, \cdots y_nu_u=z_nv_n\in \Sigma(X).\]
In the first case, clearly $\kappa\delta=\eta\varepsilon$ so that
$\bar{\bar{x}}\kappa\delta=\bar{\bar{x}}\eta\varepsilon$. Suppose therefore we have an $H(\Sigma)$-sequence as given. By  the forbidden pattern 4.4, we cannot have $x\alpha g~\rho~w\alpha^3$ and $x\alpha g~\rho~w0$ for any $w\in X$, so that $u_it_i, v_it_i\in S\backslash\{0, \alpha^3\}$ for all $1\leq i\leq n$. For each $1\leq i\leq n$, consider $y_iu_i=z_iv_i\in \Sigma(X)$. Notice first that $x\alpha g~\rho~y_iu_it_i~\rho~z_iv_it_i$. If $y_i=z_i$, then by forbidden patterns 4.4, 4.5 and 4.6, we know $(u_i)\psi=(v_i)\psi$, so that $$(y_i^{(u_i)\psi}(u_i)\phi, y_i^{(v_i)\psi}(v_i)\phi)\in P.$$ Also, as $u_it_i, v_it_i\in S\backslash \{0, \alpha^3\}$, we deduce $$(y_i^{(u_it_i)\psi}(u_i)\phi, y_i^{(v_it_i)\psi}(v_i)\phi)\in P$$ so that $$(y_i^{(u_it_i)\psi}(u_i)\phi(t_i)\phi, y_i^{(v_it_i)\psi}(v_i)\phi(t_i)\phi)\in \bar{\rho}$$

If $y_i\neq z_i$, then neither $(u_i)\psi$ or $(v_i)\psi$ is 0 by our original assumptions on
$\Sigma(X)$. By the above analysis,  $(u_i)\psi, (v_i)\psi \in \{1, 2\}$. Notice that $(u_i)\psi$ and $(v_i)\psi$ may not be equal.

It follows from the construction of $P$ that $$(y_i^{(u_i)\psi}(u_i)\phi, z_i^{(v_i)\psi}(v_i)\phi)\in P.$$ Also, as $u_it_i, v_it_i\in S\backslash \{0, \alpha^3\}$, we deduce $$(y_i^{(u_it_i)\psi}(u_i)\phi, z_i^{(v_it_i)\psi}(v_i)\phi)\in P$$  so that $$(y_i^{(u_it_i)\psi}(u_i)\phi(t_i)\phi, z_i^{(v_it_i)\psi}(v_i)\phi(t_i)\phi)\in \bar{\rho}.$$ On the other hand, as the identities involving in the above $H(\Sigma)$-sequence are from $F_S(Z)$, $$z_i^{(v_it_i)\psi}(v_i)\phi(t_i)\phi
=y_{i+1}^{(u_{i+1}t_{i+1})\psi}(u_{i+1})\phi(t_{i+1})\phi$$  for all $1\leq i\leq n-1$.  Hence $$x^1g=y_1^{(u_1t_1)\psi}(u_1)\phi(t_1)\phi~\bar{\rho}~
z_1^{(v_1t_1)\psi}(v_1)\phi(t_1)\phi
=y_2^{(u_2t_2)\psi}(u_2)\phi(t_2)\phi$$
$$~\bar{\rho}~\cdots~\bar{\rho}~z_n^{(v_nt_n)\psi}(v_n)\phi(t_n)\phi=x^1 h$$ giving  $(g, h)\in \bold{r}([x^1])$. Then either $g=h$ or there exists $n\in \mathbb{N}$ and  a $Q_1$-sequence such that $$g=p_1s_1, q_1s_1=p_2s_2, \cdots, q_ns_n=h$$ where $s_1, \cdots, s_n\in G$ and  $(p_i, q_i)\in Q_1$ for all $1\leq i\leq n$.  Since $\bar{\bar{x}}\alpha p_i=\bar{\bar{x}}\alpha q_i$ for all $1\leq i\leq n$ , we have $$\bar{\bar{x}}\alpha g=\bar{\bar{x}}\alpha p_1s_1=\bar{\bar{x}}\alpha q_1s_1=\cdots=\bar{\bar{x}}\alpha q_ns_n=\bar{\bar{x}}\alpha h,$$ so that $\bar{\bar{x}}\kappa\delta=\bar{\bar{x}}\eta\epsilon$.

{\em Subcase (i)(e)}  $\kappa\delta\in \alpha^2G$. In this case, we must have $\eta\varepsilon\in \alpha^2 G$. Let $\kappa\delta=\alpha^2 g$ and $\eta\varepsilon=\alpha^2 h$ for some $g, h\in G$. Considering  $x\kappa\delta~\rho~x\eta\varepsilon$, if the  $H(\Sigma)$-sequence connecting $x\kappa\delta$ to $x\eta\varepsilon$ does not  involve any $w\alpha^3 \in S$ for any  $w\in X$, then by a similar discussion to that of {\em  Subcase (i)(d)} , we can show $\bar{\bar{x}}\kappa\delta=\bar{\bar{x}}\eta\varepsilon$. If $x\alpha^2g~\rho~w\alpha^3$, then $x\alpha^2~\rho~w\alpha^3$ and $\bar{x}\alpha^2=a(x)=\bar{w}\alpha^3$, and so $a(x)s=a(x)$ for any $s\in G$. Therefore $$\bar{\bar{x}}\kappa\delta=\bar{\bar{x}}\alpha^2g=a(x)g=a(x)=a(x)h=\bar{\bar{x}}\alpha^2 h=\bar{\bar{x}}\eta\varepsilon.$$

\medskip

 {\em  Subcase (ii)}  $y\mu=\bar{\bar{x}}\kappa \in \Pi(Y)$ with $y\mu=x\kappa\in \Sigma(X)$, $z\nu=c\in \Sigma(Y)$. Suppose that  $y\mu\delta~\rho'~z\nu\varepsilon$ for some $\delta, \varepsilon\in S$. Then $x\kappa\delta~\rho~z\nu\varepsilon$, giving $\bar{x}\kappa\delta=c\varepsilon\in A$, so that $\kappa\delta=\alpha^2g$ for some $g\in G$. Since $\bar{\bar{x}}\alpha^2=a(x)=\bar{x}\alpha^2$, we have $$c\varepsilon=\bar{x}\kappa\delta=\bar{x}\alpha^2g=a(x)g=\bar{\bar{x}}\alpha^2g=\bar{\bar{x}}\kappa\delta.$$

{\em Subcase (iii)} $y\mu=b, z\nu=c\in \Sigma(Y)$. If $y\mu\delta~\rho'~z\nu\varepsilon$ for some $\delta, \varepsilon\in S$, then $y\mu\delta~\rho~z\nu\varepsilon$, giving $b\delta=c\varepsilon$ by Proposition \ref{prop:soln}.

\medskip

Therefore, $\Pi(Y)$ is consistent. By induction, $\Pi(Y)$ has a solution $(\bar{\bar{y}})_{y\in Y}$ in $A$ and hence $(\bar{\bar{y}})_{y\in X}$ is a solution to $\Sigma(X)$ in $A$.

\medskip

{\em Case $\tau(x)=\alpha $.}   We therefore have $x\alpha=a(x) \in \Sigma(X), a(x)0\neq a(x).$ Notice that, for any $x\alpha^i g\in S$ with $g\in G$ and $i\geq 1$, $$\bar{x}\alpha^ig=\bar{x}\alpha \alpha^{i-1}g=a(x)\alpha^{i-1}g\in A.$$ Let $Y=X\setminus \{ x\}$. By replacing $x$ by $\bar{x}$ in all equations of $\Sigma(X)$ involving $x$, we obtain a finite consistent set of equations  $\Sigma(Y)$  with a solution $(\bar{y})_{y\in Y},$ and so it has a solution $(\bar{\bar{y}})_{y\in Y}$ in $A$ by our inductive hypothesis. Further, the set of equations
$\Sigma(x)$ of $\Sigma(X)$ which involve only the variable $x$ has a solution $\bar{\bar{x}}\in A$. We claim that  $(\bar{\bar{y}})_{y\in X}$ is a solution to $\Sigma(X).$ To this end we need only check equations of the form $x\beta=y\gamma$. By assumption, $\beta=\alpha \delta$ for some $\delta\in S$ and then
\[\bar{x}\alpha=a(x)=\bar{\bar{x}}\alpha\] so that
$\bar{\bar{y}}\gamma=\bar{x}\beta=\bar{\bar{x}}\beta$, as required.

\medskip

This concludes the proof that every almost pure $S$-act over the Fountain monoid  is absolutely pure.\end{proof}

\begin{qn}\label{qn:moreexamples}  Let $S$ be the  monoid  obtained by replacing the group $G$ in Example \ref{fountain} with any right coherent monoid $T$ such that the universal right congruence $\omega_T$ on $T$ is not finitely generated (for example, any monoid semilattice without a zero \cite{dandan:2019}). For such an $S$, the same argument as in \cite{gould:1992} gives that $S$ is not right coherent. However, can we deduce  $\mathcal{A}_S(1)=\mathcal{A}_S(\aleph_0)$? What if we change the period of $\alpha$? More speculatively, are all monoids such that $\mathcal{A}_S(1)=\mathcal{A}_S(\aleph_0)$ built in some way from right coherent monoids, and monoids satisfying the fem-property?
\end{qn}

\section{Condition for when almost pure acts are absolutely pure}\label{sec:ap=ap}

Let $\mathscr{G}$ be a set of finite frames and let $\F\subseteq \mathscr{G}$.
We first give a generic result that tells us when all $\F$-pure acts are $\mathscr{G}$-pure. We then specialise this to Theorem~\ref{main theorem}, which gives a  condition for all almost pure $S$-acts to be absolutely pure entirely in terms of finitely presented $S$-acts, their $S$-subacts, and their canonical extensions. Recall that the canonical extensions are obtained via analysis of which sets of equations are consistent, which is itself described in terms of congruences on certain free $S$-acts and a `base' $S$-act.

\begin{Thm}\label{almostmaintheorem} Let $S$ be a monoid and let $\F\subseteq \mathscr{G}$, where $\mathscr{G}$ is a set of finite frames. The following are equivalent:

 \begin{enumerate} \item every $\F$-pure  $S$-act is $\mathscr{G}$-pure;
 \item every $S$-act of the form $A(\F)$ is $\mathscr{G}$-pure;
 \item for any $S$-act $A$, we have that $A(\F)$ is a retract of $A(\mathscr{G})$.
 \end{enumerate}
\end{Thm}
\begin{proof} Clearly (1) implies (2).
Suppose now that  (2) holds. In view of our careful constructions, we may regard $A(\mathscr{G})$ as being built from $A(\F)$.  For, having constructed $A^{\mathscr{G}}_{i}$ from $A(\F)$, obtaining an amalgam $A^{\mathscr{G}}_{i}\cup A(\F)$, with $A^{\mathscr{G}}_{i}\cap A(\F)=A^{\mathscr{F}}_{i}$,we construct $A^{\mathscr{G}}_{i+1}$ from this amalgam  by  making only those extensions that add  in solutions to finite consistent sets  of $\mathscr{G}$-equations with frames in $\mathscr{G}\setminus \F$, or finite consistent sets of $\mathscr{G}$-equations that have constants in $A^{\mathscr{G}}_{i}\cap A(\F)$. The way in which we always choose new variables to build our extensions ensures no contradiction arises.
   Proposition~\ref{prop:built} now gives that (3) holds.

  Finally, suppose that  (3) holds  and $A$ is an $\F$-pure $S$-act.  By Theorem~\ref{thm:towers} we have that $A$ is a retract of $A(\F)$, so that by (3), $A$ is a retract of $A(\mathscr{G})$.  A second application of Theorem~\ref{thm:towers} yields (1).
\end{proof}

Theorem~\ref{almostmaintheorem} is to a certain extent a universal-type result. The following is saying something more, and highlights the connection between finitely generated subacts of finitely presented $S$-acts (hence, coherency), and the question of when every almost pure $S$-act is absolutely pure.

Let $\mathscr{G}$ be the set of all finite frames, and let $\F$ be the set of all finite $1$-frames. For an $S$-act $A$ we let $A(\aleph_0):= A(\mathscr{G})$ and $A(1)=:A(\F)$.

 \begin{Thm} \label{main theorem} The following are equivalent for a monoid $S$:

 \begin{enumerate} \item every almost pure $S$-act is absolutely pure;
 \item every $S$-act of the form $A(1)$ is absolutely pure;
 \item every $S$-act of the form $A(1)$ where $A$ is a finitely generated $S$-subact of a finitely presented $S$-act is absolutely pure;
 \item for any $S$-act $A$ we have that $A(1)$ is a retract of $A(\aleph_0)$.
 \item for any $S$-act $A$, where $A$ is a finitely generated $S$-subact of a finitely presented $S$-act, $A(1)$ is a retract of $A(\aleph_0)$.
 \end{enumerate}
 \end{Thm}
 \begin{proof} The equivalence of (1), (2) and (4) follows from Theorem~\ref{almostmaintheorem}. Clearly (2) implies (3) and (4) implies (5). Proposition~\ref{prop:retractoneway} and Theorem~\ref{thm:towers} give that  (5) implies (3).

Suppose that (3) holds.   Let $A$ be an almost pure $S$-act and let $\theta:B\rightarrow A$ be an $S$-morphism, where $B$ is a finitely generated $S$-subact of a finitely presented $S$-act $M$.  We follow the proof of Lemma~\ref{lem:building} to obtain an $S$-morphism $\theta_1:B^1_1\rightarrow A$ extending $\theta$.  But then we can iterate this process to obtain an $S$-morphism $\varphi:B(1)\rightarrow A$ extending
 $\theta$.

 Now, $B$ is embedded in $M$ where $M$ is finitely presented. Suppose that
 $M=F_S(X)/\rho$ where $X$ is finite and $\rho=\langle H\rangle$ where $H$ is finite.
 Let $C$ be a set of generators for $B$ and for each generator $c\in C$ pick $x_cs_c\in F_S(X)$ so that
 $c=[x_cs_c]$ and
 let \[\Sigma=\{ xu=yv, x_cs_c=c: (xu,yv)\in H, c\in C\}.\]
 Clearly, $\Sigma$ has a solution in $M$ where we substitute $[x]$ for $x$ for each $x\in X$. Regard $\Sigma$ as a set of equations over $B(1)$. Our assumption is that $B(1)$ is absolutely pure, so that there exists a solution $(c_x)_{x\in X}$ to $\Sigma$ in  $B(1)$.
 A standard argument then gives that
 $\psi:M\rightarrow B(1)$ given by $[x]\psi=c_x$ is a well defined $S$-morphism.
 Let $d\in B$, so that $d=cs$ for some $c\in C$. Then
 $$d\psi=(cs)\psi=(c\psi) s=[x_cs_c]\psi s=[x_c]\psi s_cs=c_{x_c}s_cs=cs=d.$$ Now consider
 $\psi\varphi:M\rightarrow A$. Clearly $\psi\varphi$ is an $S$-morphism, and
 $d\psi\varphi=d\varphi=d\theta$, for any $d\in B$. It follows from
 Theorem~\ref{thm:connection} that $A$ is absolutely pure.
 This completes the proof of (3) implies (1).
\end{proof}

Given the results of this article one might ask whether it true that for any monoid $S$ we have $\mathcal{A}_S(1)=\mathcal{A}_S(\aleph_0)$?  We conjecture that this is not the case and would hope that Theorem~\ref{main theorem} would help in constructing a counter-example.

\section*{Acknowledgements} The results of Subsection~\ref{subs:fem}  essentially follow from observations obtained in discussions with Professor Nik Ru\v{s}kuc. The authors are grateful for his permission to use them in this article.

\end{document}